\documentclass[preprint]{elsarticle}

\usepackage{lineno}

\modulolinenumbers[5]

\journal{Advances in Mathematics}
\usepackage{amsmath}
\usepackage{amsmath,amscd}
\usepackage{amsfonts}
\usepackage{amssymb}
\usepackage{amsthm}
\DeclareMathAlphabet{\pazocal}{OMS}{zplm}{m}{n}
\usepackage{cancel}

\usepackage[all,cmtip]{xy}

\usepackage[normalem]{ulem}
\usepackage{wrapfig}
\usepackage{pb-diagram}
\usepackage{mathabx}
\usepackage{lscape}
\usepackage{geometry}

\newtheorem{defi}{\textbf{Definition}}[section]

\newtheorem{teo}[defi]{\textbf{Theorem}}
\newtheorem{cor}[defi]{\textbf{Corollary}}
\newtheorem{prop}[defi]{\textbf{Proposition}}
\newtheorem{lem}[defi]{\textbf{Lemma}}
\newtheorem{rem}[defi]{\textbf{Remark}}

\newcommand{\dbl}{[\hspace{-0.2ex}[}
\newcommand{\dbr}{]\hspace{-0.2ex}]}
\newcommand{\db}[1]{\dbl {#1} \dbr}
\newcommand{\res}[1]{\hspace{-0.6mm}\downarrow_{\hspace{-0.25mm}{#1}}}
\newcommand{\ind}[1]{\hspace{-0.6mm}\uparrow^{\hspace{-0.25mm}{#1}}}

\newcommand{\ctens}{\widehat{\otimes}}

\newcommand{\iso}{\cong}
\newcommand{\invlim}{\underleftarrow{\textnormal{lim}}\,}

\newcommand{\ds}{\raisebox{0.5pt}{\,\big|\,}}

\newcommand{\tn}[1]{\textnormal{#1}}

\begin{document}

\begin{frontmatter}

\title{Block theory and Brauer's first main theorem for profinite groups\tnoteref{Block theory and Brauer's first main theorem for profinite groups}} %Ricardo: Suspeito que o título curto fica um pouco longo agora??



\author[ufmg]{Ricardo J. Franquiz Flores\corref{cor1}\fnref{fn1}}
\ead{rfranquiz@ufmg.br}
\author[ufmg]{John W. MacQuarrie}%\fnref{fn2}}
\ead{john@mat.ufmg.br}

\cortext[cor1]{Corresponding author}
\fntext[fn1]{The author was financially supported by a CAPES doctoral scholarship.}
%\fntext[fn2]{Another author footnote, but a little more longer.}

\address[ufmg]{Universidade Federal de Minas Gerais, Belo Horizonte, Brazil}

\begin{abstract}
We develop the local-global theory of blocks for profinite groups. Given a field $k$ of characteristic $p$ and a profinite group $G$, one may express the completed group algebra $k\db{G}$ as a product $\prod_{i\in I}B_i$ of closed indecomposable algebras, called the blocks of $G$.  To each block $B$ of $G$ we associate a pro-$p$ subgroup of $G$, called the defect group of $B$, unique up to conjugacy in $G$.  We give several characterizations of the defect group in analogy with defect groups of blocks of finite groups.  Our main theorem is a version of Brauer's first main theorem: a correspondence between the blocks of $G$ with defect group $D$ and the blocks of the normalizer $\tn{N}_G(D)$ with defect group $D$.
\end{abstract}


\end{frontmatter}

%\linenumbers

\section{Introduction}
\label{sec:intro}
Let $G$ be a finite group and $k$ a field of characteristic $p$. The algebra $kG$, as with any finite dimensional associative algebra, has a unique decomposition as a direct product of indecomposable algebras $kG = B_1\times \cdots \times B_n$, the $B_i$ being known in group representation theory as the \emph{blocks} of $G$.  Rather than try to deal with $kG$ all at once, the philosophy of block theory is to study the blocks of $G$ individually, as the representation theory of certain $B_i$ may be considerably more tractable than the representation theory of $kG$.

The difficulty of a block $B$ is measured by a $p$-subgroup of $G$ (unique up to $G$-conjugacy) called the \emph{defect group} of $B$.  As examples of this ``measure'': blocks with cyclic defect group are completely understood, while there are structural conjectures about blocks whose defect group is abelian that are false for arbitrary blocks, perhaps the most famous being Brou\'e's Abelian Defect Group Conjecture (for a light-hearted survey of the state-of-the-art in modular representation theory and block theory see \cite{Craven}).

The most fundamental theorem of the block theory of finite groups is Brauer's first main theorem (see \cite[Theorem 5]{brauer1944} for the original statement or for instance \cite[Theorem 6.2.6]{B} for a more modern one), which gives a natural bijection between the blocks of $G$ with defect group $D$ and the blocks of the normalizer $\tn{N}_G(D)$ of $D$ in $G$ with defect group $D$.  Our main theorem is Brauer's first main theorem for profinite groups:

\begin{teo}
Let $k$ be a field of characteristic $p$, $G$ a profinite group and $D$ a closed subgroup of $G$ that is open in a $p$-Sylow subgroup that contains it. There is a natural bijection between the blocks of $G$ over $k$ with defect group $D$, and the blocks of $\tn{N}_{G}(D)$ over $k$ with defect group $D$.
\end{teo}
The hypothesis that $D$ be open in a $p$-Sylow subgroup is necessary, but is not a restriction because we will show that the defect group of every block of $G$ has this property.

More broadly, we begin a systematic study of the local-global theory of blocks of a profinite group $G$ over a field $k$ of characteristic $p$.  We define the defect group of a block of $G$, prove its basic properties and give several alternative characterizations (Theorem \ref{theorem:CompleteDefectGroupChar}), before proving Brauer's first main theorem.  

The most striking feature of the discussion is the generality that block theory permits: 
all results apply to arbitrary profinite groups.  This suggests that the block theoretic approach to the representation theory of profinite groups will be as central to the theory as it is with finite groups.  For comparison, we note that with modules, conditions on the group are frequently required to prove the most powerful results -- see for instance \cite{J1,J}.

Because we do not demand that the field $k$ be finite, our algebras and modules are not compact, but they are \emph{pseudocompact}, thus maintaining a weaker version of compactness (linear compactness) that suffices for our purposes.  While pseudocompact objects are perhaps less familiar than compact ones, the advantage of allowing this greater generality is undeniable, as it gives a representation theorist access to algebraically closed fields.

In a future paper, the tools developed here will be used to provide a complete classification of the blocks of a profinite group having (finite or infinite) cyclic defect group, in close analogy with the classification of blocks of finite groups with cyclic defect group.

The background to the material discussed in this article is well served by the literature.  For an introduction to profinite groups see \cite{ZR} or \cite{W}.  For basic properties of pseudocompact algebras see \cite{Bru} or \cite{PPJ}.   For introductory expositions on the modular representation theory of finite groups see \cite{B} or \cite{A}.  For an excellent and encyclopedic introduction to the block theory of finite groups see \cite{Lin1, Lin2}.

\section*{Acknowledgements:}

We thank Jeremy Rickard, Peter Symonds and Pavel Zalesskii for helpful conversations, and the anonymous referee for pertinent comments.

\section{Preliminaries}
\label{sec:prelim}
\subsection{Pseudocompact algebras and modules}

Let $k$ be a discrete topological field of characteristic $p$.  In what follows, when the coefficient ring is unspecified it is assumed to be $k$ (so for example ``algebra'' means ``$k$-algebra'').  Throughout this section, unless otherwise specified, all modules are topological left modules. For a general introduction to pseudocompact objects, see \cite{Bru}.

\begin{defi}
	A \emph{pseudocompact algebra} is an associative, unital, Hausdorff topological $k$-algebra $A$ possessing a basis of $0$ consisting of open ideals $I$ having cofinite dimension in $A$ that intersect in 0 and such that $A=\varprojlim_{I}A/I$.
\end{defi}

Equivalently, a pseudocompact algebra is an inverse limit of discrete finite dimensional algebras in the category of topological algebras.  If $G$ is a profinite group, an inverse system of finite continuous quotients $G/N$ of $G$ induces an inverse system of finite dimensional algebras $k[G/N]$, whose inverse limit $k\db{G}$, the \emph{completed group algebra} of $G$, is a pseudocompact algebra.  

We will occasionally make use of the Jacobson radical $J(A)$ of a pseudocompact algebra $A$: that is, the intersection of the maximal open left ideals of $A$.  Essentially the only property we require of the Jacobson radical is the easily checked fact that if $\rho:A\to B$ is a surjective homomorphism of pseudocompact algebras, then $\rho(J(A))\subseteq J(B)$ (cf.\ \cite[\S 1]{Bru} for more on $J(A)$).

If $A$ is a pseudocompact algebra, a \emph{pseudocompact $A$-module} is a topological $A$-module $U$ possessing a basis of 0 consisting of open submodules $V$ of finite codimension that intersect in 0 and such that $U=\varprojlim_{V} U/V$. The category of pseudocompact modules for a pseudocompact algebra has exact inverse limits.

If $V$ is a topological vector space, cosets of closed subspaces of $V$ are called \emph{closed linear varieties}.  We say that $V$ is \emph{linearly compact} if for every family $\mathcal{F}=\{W_{i}\mbox{ : } i\in I\}$ of closed linear varieties of $V$ with the finite intersection property, we have that $\bigcap_{i\in I}W_{i}\neq\emptyset$.

\begin{lem}\label{JK-lem}
	Let $A$ be a pseudocompact algebra and $U,V$ pseudocompact $A$-modules.
	\begin{enumerate}
		\item The module $U$ is linearly compact.
		\item If ever $\rho: U\to V$ is a continuous homomorphism, then $\rho(U)$ is linearly compact and hence closed in $V$.
		\item The submodule abstractly generated by a finite subset of $U$ is closed.
		\item If $U$ is finitely generated as an $A$-module, then every $A$-module homomorphism $U\to V$ is continuous.
	\end{enumerate}
\end{lem}

\begin{proof}
	Items 1 to 3 are \cite[Lemma 2.2]{JK} and Item 4 is \cite[Proposition 3.5]{VanGas}.
\end{proof}

We collect here several properties of linearly compact vector spaces, given in \cite[II, \S6.27]{L}, which will be used frequently:

\begin{lem}\label{lem:propertiesofLCspaces}
	\begin{enumerate}
		\item Discrete finite dimensional $k$-vector spaces are linearly compact. 
		\item A product of linearly compact $k$-vector spaces is linearly compact.
		\item Let $\{V_{i},\varphi_{ij},I\}$ be an inverse system of topological vector spaces and for each $i$, let $W_{i}$ be a non-empty closed linear variety in $V_{i}$.  If $\varphi_{ij}(W_{i})\subseteq W_{j}$, so that $\{W_{i},\varphi_{ij},I\}$ forms an inverse system, then its inverse limit is non-empty. 
		\item A continuous linear bijective map between two linearly compact spaces is an isomorphism.
	\end{enumerate}
\end{lem}

Several times in this section we will give results for pseudocompact algebras and cite for the proof a reference whose statement is for profinite algebras (that is, compact rather than just linearly compact algebras). When we do this, what we mean is that the proof in the given reference goes through without change for pseudocompact algebras.

\subsection{Induction and restriction}

The following definition makes use of the \emph{completed tensor product} (cf.\ \cite[\S 2.2]{PPJ}).  Let $G$ be a profinite group and $H$ a closed subgroup of $G$. If $V$ is a pseudocompact $k\db{H}$-module, then the \emph{induced $k\db{G}$-module} is defined by $V\uparrow^{G}=k\db{G}\widehat{\otimes}_{k\db{H}}V$, with action from $G$ on the left factor. If $U$ is a $k\db{G}$-module, then the \emph{restricted $k\db{H}$-module} $U\downarrow_{H}$ is the original $k\db{G}$-module $U$ with coefficients restricted to the subalgebra $k\db{H}$. Induction is left (but not right) adjoint to restriction \cite[\S 2.2]{PPJ}. If $H$ is a closed subgroup of $G$, denote by $k\db{G/H}$ the left pseudocompact permutation $k\db{G}$-module given as the inverse limit of the modules $k[G/HN]$, as $N$ runs through the open normal subgroups of $G$ (cf.\ \cite[\S 2.3]{PPJ}).

\begin{lem}(\cite[Proposition 5.8.1]{ZR})\label{lem:isoind/res}
	Let $G$ be a profinite group, $H$ a closed subgroup of $G$ (denoted $H\leq_{C}G$) and $U$ a pseudocompact $k\db{G}$-module. The map $g\ctens_{k\db{H}} u\mapsto gH\ctens_k gu$ induces a natural continuous
	isomorphism of left $k\db{G}$-modules
	$$k\db{G}\widehat{\otimes}_{k\db{H}}U\cong k\db{G/H}\widehat{\otimes}_{k}U,$$
	
	where the action of $k\db{G}$ on $k\db{G/H}\widehat{\otimes}_{k}U$ is diagonal.
\end{lem}

\begin{lem}\label{lem:invmodind/red}
	Let $G$ be a profinite group, $H\leq_{C}G$ and $U$ a pseudocompact $k\db{G}$-module. Then $U\res{H}\ind{G}\cong\varprojlim_{N}U\res{HN}\ind{G}$.
\end{lem}

\begin{proof}
	This follows from \cite[Proposition 5.2.2, Lemma 5.5.2, Proposition 5.8.1]{ZR}.
\end{proof}

\subsection{Homomorphisms and coinvariants}

Let $U$ be a pseudocompact $k\db{G}$-module and $N$ a closed normal subgroup of $G$.  The module of \emph{$N$-coinvariants} $U_{N}$ is defined as $k\widehat{\otimes}_{k\db{N}} U\cong U/I_{N}U$, where $I_{N}$ denotes the augmentation
ideal of $k\db{N}$ -- that is, the kernel of the continuous map $k\db{N}\to k$ sending $n\mapsto 1$ for each $n\in N$. The action of $G$ on $U_N$ is given by $g(\lambda\widehat{\otimes}u) = \lambda\widehat{\otimes}gu$.  The subgroup $N$ acts trivially on $U_N$, so we can treat $U_N$ as a $k[G/N]$-module if we wish.  Properties of the completed tensor product imply that $(-)_N$ is a right exact functor from the category of finitely generated $k\db{G}$-modules to the category of finitely generated $k\db{G/N}$-modules.  If $U$ is a finitely generated pseudocompact $k\db{G}$-module and $N$ is an open normal subgroup of $G$, denoted $N\trianglelefteq_{O} G$, then $U_N$ is finite dimensional.

The module $U_{N}$, together with the canonical quotient map $\varphi_N:U\to U_N$, satisfies the following universal property:

\emph{Every continuous $k\db{G}$-module homomorphism $\rho$ from $U$ to a pseudocompact $k\db{G}$-module $X$ on which $N$ acts trivially, factors uniquely through $\varphi_{N}$. That is, there is a unique continuous homomorphism $\rho':U_{N}\to X$ such that $\rho=\rho'\varphi_{N}$.}

We collect here several technical properties of coinvariant modules:

\begin{lem}(\cite[Lemma 2.6]{J})\label{lem:propertiescoinvtechn}
	Let $G$ be a profinite group, $N,M$ closed normal subgroups of $G$
	with $N\leq M$ and $H$ a closed subgroup of $G$.  Let $U,W$ be
	pseudocompact $k\db{G}$-modules and let $V$ be a $k\db{H}$-module.  Then
	
	\begin{enumerate}
		\item $(U_N)_M$ is naturally isomorphic to $U_M$.
		\item $(U\oplus W)_N\cong U_N\oplus W_N$.
		\item $V_{H\cap N}$ is naturally a $k\db{HN/N}$-module.
		\item $(V\uparrow^{G})_N\cong V_{H\cap N}\uparrow^{G/N}$.
		\item $U_N\downarrow_{HN/N}\cong (U\downarrow_{HN})_N$.
	\end{enumerate}
\end{lem}

\begin{prop}(\cite[Proposition 2.7]{J})\label{prop:invlimcoinv}
	If $U$ is a pseudocompact $k\db{G}$-module, then $\{U_{N} \,:\, N\trianglelefteq_{O} G\}$ together with the set of canonical quotient maps forms a surjective inverse system with inverse limit $U$.
\end{prop}

\begin{rem}\label{obs:invsystemforcoinvariants}
	Throughout this text, whenever $N\leqslant M$ are closed normal subgroups of the profinite group $G$, the notations $\varphi_N, \varphi_{MN}$ will be reserved exclusively for the canonical maps between coinvariant modules.  In the special case of $k\db{G}$ we have $k\db{G}_N = k\db{G/N}$ and the corresponding maps $\varphi_N, \varphi_{MN}$ are in fact algebra homomorphisms.
\end{rem}

Let $A$ be a pseudocompact algebra. If $U$ and $W$ are topological $A$-modules, then $\tn{Hom}_{A}(U, W)$ denotes the topological $k$-vector space of continuous $A$-module homomorphisms from $U$ to $W$ with the compact-open topology. If $U,W$ are pseudocompact, and $W=\varprojlim_{i}W_{i}$, then $\tn{Hom}_{A}(U,W)=\varprojlim_{i}\tn{Hom}_{A}(U,W_{i})$. In particular, when $U$ is finitely generated as an $A$-module, then $\tn{Hom}_{A}(U,W)$ is a pseudocompact vector space. For more details see \cite[\S 2.2]{PPJ}.  In the special case where $A = k\db{G}$, the natural isomorphism $\tn{Hom}_{k\db{G}}(U,W_N)\cong \tn{Hom}_{k\db{G}}(U_N,W_N)$ implies that $\tn{Hom}_{k\db{G}}(U,W)\cong \varprojlim_{N}\tn{Hom}_{k\db{G}}(U_N,W_N)$.

\begin{prop}\label{prop:localEndring}
	Let $G$ be a profinite group and let $U$ be an indecomposable finitely generated $k\db{G}$-module. Then $\tn{End}_{k\db{G}}(U)$ is a local pseudocompact algebra.
\end{prop}

\begin{proof}
	The algebra $\tn{End}_{k\db{G}}(U)$ is pseudocompact by \cite[Lemma 2.3]{PPJ}.  By \cite{dimitric}, 
	linearly compact modules are algebraically compact (see for instance \cite[Proposition 4.11]{dimitric}) and so the ring of abstract endomorphisms of $U$ is local \cite[Theorem 4.15 and Proposition 4.10]{dimitric}.  But since $U$ is finitely generated, the rings of abstract and continuous endomorphisms coincide by Lemma \ref{JK-lem}.
\end{proof}

\begin{lem}\label{lem:splitmorphisms}
	Let $U, W$ be pseudocompact $k\db{G}$-modules with $U$  finitely generated. If $\pi:W\to U$ is a surjective $k\db{G}$-module homomorphism such that the induced map $\pi_{N}:W_{N}\to U_{N}$ splits as a $k[G/N]$-module homomorphism for each $N\trianglelefteq_{O}G$, then $\pi$ splits as a $k\db{G}$-module homomorphism.
\end{lem}

\begin{proof}
	For each $N\trianglelefteq_{O}G$ we have by hypothesis a non-empty set $X_N\subseteq \tn{Hom}_{k\db{G}}(U_N,W_N)$ of splittings of $\pi_N$.  For $N\leqslant M$ the map $X_N\to X_M$ given by $\alpha\mapsto \alpha_M$ makes the $X_N$ into an inverse system.  Each $X_N$ is a non-empty closed linear variety, being equal to $\alpha + \tn{Ker}(\Gamma)$, where $\alpha$ is an arbitrary element of $X_N$ and $\Gamma:\tn{Hom}_{k\db{G}}(U_N,W_N)\to \tn{End}_{k\db{G}}(U_N)$ is the continuous linear map sending $\gamma\mapsto \pi_N\gamma$.  It follows from Lemma \ref{lem:propertiesofLCspaces} that $\varprojlim X_N\neq\varnothing$.  An element $\iota\in \varprojlim X_N$ is a splitting of $\pi$.
\end{proof}

\begin{rem}
	Lemma \ref{lem:splitmorphisms} may be compared with  \cite[Lemma 3.4]{J}, which says that when $k$ is finite and $U,W$ are finitely generated, then if $U_N$ is isomorphic to a direct summand of $W_N$ for each $N$, then $U$ is isomorphic to a direct summand of $W$.  On one hand Lemma \ref{lem:splitmorphisms} is stronger, as it does not require that $W$ be finitely generated.  On the other hand, even when $W$ is assumed to be finitely generated, the proof of \cite[Lemma 3.4]{J} does not obviously carry through when $k$ is infinite.  This weaker version (wherein we must specify our candidate $\pi$ for the split surjection) will be sufficient for our purposes.  If $W$ is assumed to be finitely generated, one may alternatively fix the candidate $\iota:U\to W$ for a split injection and find a corresponding split surjection $\pi$.
\end{rem}

\subsection{Relative Projectivity}

Most of the results of this section were proved for $k$ finite in \cite{J}.  We restate them here in the greater generality we require, observing any modifications required in the proof.  A pseudocompact $k\db{G}$-module $U$ is \emph{relatively $H$-projective} if given any diagram of pseudocompact $k\db{G}$-modules and continuous $k\db{G}$-module homomorphisms of the form
\begin{eqnarray*}
	\xymatrix{
		& U   \ar[d]^{\varphi} \\
		V \ar@{->>}[r]_{\beta} & W }
\end{eqnarray*}
such that there exists a continuous $k\db{H}$-module homomorphism $ U\to V$ making the triangle commute, then there exists a continuous $k\db{G}$-module homomorphism $U\to V$ making the triangle commute. Observe that a projective module is the same thing as a 1-projective module.  Several familiar characterizations of relatively projective modules follow exactly as with finite groups:

\begin{lem}\label{lem:basicrelprojchar}
	Let $G$ be a profinite group, $H$ a closed subgroup of $G$ and let $U$ be a pseudocompact $k\db{G}$-module.  The following are equivalent.
	\begin{enumerate}
		\item $U$ is relatively $H$-projective,
		\item If a continuous $k\db{G}$-module epimorphism $V\to U$ splits as a $k\db{H}$-module homomorphism, then it splits,
		\item\label{item:UsummandUresind} $U$ is isomorphic to a direct summand of $U\res{H}\ind{G}$,
		\item\label{item:projectionsplits} The natural projection $\pi:U\res{H}\ind{G}\to U$ sending $g\ctens u\mapsto gu$ splits.
		\item $U$ is isomorphic to a direct summand of $X\ind{G}$ for some $k\db{H}$-module $X$.
	\end{enumerate}
\end{lem}

\begin{proof}
	This goes through as for finite groups (cf.\ \cite[Proposition 3.6.4]{B}).  We show only that \ref{item:UsummandUresind} implies \ref{item:projectionsplits}, since it is not completely standard (though still well-known).  Note that $\pi$ is the component at $U$ of the counit of the induction-restriction adjunction. If $\alpha:U\to U\res{H}\ind{G}$ and $\beta:U\res{H}\ind{G}\to U$ are the given splitting maps then, by formal properties of adjoint functors, $\beta = \pi\circ \gamma\ind{G}$ for some endomorphism $\gamma$ of $U\res{H}$.  But then $\tn{id}_U = \beta\alpha = \pi (\gamma\ind{G}\alpha)$, showing that $\pi$ is split.
\end{proof}

\begin{prop}\label{prop:relprojHN}
	Let $U$ be a finitely generated pseudocompact $k\db{G}$-module, and $H \leq_{C} G$. Then $U$ is relatively $H$-projective if, and only if, $U$ is relatively $HN$-projective for every $N \trianglelefteq_{O} G$.
\end{prop}

\begin{proof}
	One may follow the proof of \cite[Proposition 3.3]{J}, noting as in the proof of Lemma \ref{lem:splitmorphisms} that the subsets $I_N$ of $\tn{Hom}_{k\db{G}}(U,U\res{HN}\ind{G})$ are closed linear varieties, so that the inverse limit of the $I_N$ is non-empty by Lemma \ref{lem:propertiesofLCspaces}.
\end{proof}

Recall that a subset $J$ of the directed set $I$ is \emph{cofinal} if, for every $i\in I$ there is $j \in J$ with $j\geqslant i$.  In this case whenever $\{X_i\,|\,i\in I\}$ is an inverse system indexed by $I$, then $\{X_i\,|\,i\in J\}$ is also an inverse system and $\invlim_{i\in I}X_i = \invlim_{i\in J}X_i$.  Our directed set will usually be the set of open normal subgroups of a profinite group $G$, ordered by reverse inclusion, so that cofinal means that for each open normal subgroup $M$ of $G$, there is $N$ in the cofinal subset with $N\leqslant M$.

\begin{prop}\label{prop:charrelprojcoinv}
	Let $U$ be a finitely generated pseudocompact $k\db{G}$-module, and $H \leq_{C} G$. Then $U$ is relatively $H$-projective if, and only if, $U_{N}$ is relatively $HN$-projective for every $N\trianglelefteq_{O} G$.
\end{prop} 

\begin{proof}
	We essentially follow the proof of \cite[Proposition 3.5]{J}.  Fix $M\trianglelefteq_O G$ and consider the cofinal system of open normal subgroups $N$ of $G$ contained in $M$.  The module $U_N$ is relatively $HM$-projective and hence by Lemma \ref{lem:basicrelprojchar}, the canonical projection $U_N\res{HM}\ind{G}\to U_N$ splits.  But $U_N\res{HM}\ind{G} \iso (U\res{HM}\ind{G})_N$ and hence by Lemma \ref{lem:splitmorphisms}, $U$ is relatively $HM$-projective.  The result now follows from Proposition \ref{prop:relprojHN}.
\end{proof}

\begin{defi}\label{def:tracmap}
	If $H\leq_{O} G$ and $U$, $W$ are $k\db{G}$-modules, then \emph{the trace map}

	$$\tn{Tr}_{H}^{G}: \tn{Hom}_{k\db{H}}(U\downarrow_{H}, W\downarrow_{H})\to \tn{Hom}_{k\db{G}}(U, W),$$
	is defined by $\alpha \mapsto\sum_{s\in G/H}s\alpha s^{-1}$.
\end{defi}

Collecting the above results and mimicking the finite case for the final item, we have:

\begin{teo}\label{teo:projrel}
	Let $G$ be a profinite group, $H\leq_{C} G$, and $U$ a finitely generated pseudocompact $k\db{G}$-module. The following are equivalent:
	\begin{enumerate}
		\item $U$ is relatively $H$-projective,
		\item If a continuous $k\db{G}$-module epimorphism $V\to U$ splits as a $k\db{H}$-module homomorphism, then it splits,
		\item $U$ is isomorphic to a continuous direct summand of $U\res{H}\ind{G}$,
		\item The natural projection $\pi:U\res{H}\ind{G}\to U$ sending $g\ctens u\mapsto gu$ splits.
		\item $U$ is isomorphic to a continuous direct summand of $X\ind{G}$ for some $k\db{H}$-module $X$.
		\item $U$ is relatively $HN$-projective for every $N\trianglelefteq_{O} G$.
		\item $U_N$ is relatively $HN$-projective for every $N\trianglelefteq_{O} G$.
		\item For every $N\trianglelefteq_{O} G$ there exists a continuous $k\db{HN}$-endomorphism $\alpha_{N}$ of $U$ such that $\tn{id}_{U} = \tn{Tr}_{HN}^{G}(\alpha_{N})$.
	\end{enumerate}
\end{teo}

\subsection{Vertices}

Let $G$ be a profinite group and $U$ a finitely generated indecomposable pseudocompact $k\db{G}$-module.  A \emph{vertex} of $U$ is a closed subgroup $Q$ of $G$ such that $U$ is projective relative to $Q$, but not projective relative to any proper closed subgroup of $Q$.

The following results were proved for profinite $k\db{G}$-modules in \cite[Section 4]{J}, but both statements and proofs carry through for pseudocompact $k\db{G}$-modules, given Proposition \ref{prop:localEndring}.

\begin{prop}\label{lem:existvetx}\label{teo:verticesareconjugate}\label{lem:vtxpro-psubgrp}
	(\cite[Section 4]{J})
	If $U$ is an indecomposable finitely generated pseudocompact $k\db{G}$-module, then a vertex of $U$ exists.  The vertices of $U$ are a $G$-conjugacy class of pro-$p$ subgroups of $G$.
\end{prop}

Here and elsewhere, when $U,V$ are $k\db{G}$-modules, we write $U\ds V$ to mean that $U$ is isomorphic to a continuous direct summand of $V$.

\begin{lem}\label{lem:vtxsimplemodule}
	Let $G$ be a profinite group, $N\trianglelefteq_{C}G$ and $U$ a finitely generated indecomposable $k\db{G}$-module on which $N$ acts trivially. If $U$ has vertex $Q$ as a $k\db{G}$-module, then $U$ has vertex $QN/N$ as a $k[G/N]$-module.
\end{lem}

\begin{proof}
	The module $U$ has vertex $R/N$ as a $k[G/N]$-module, for some $N\leq R\leq QN$.  Note that $N$ acts trivially on $U\res{R}\ind{G}$, so that by Lemma \ref{lem:propertiescoinvtechn}
	$$U\res{R}\ind{G}\cong (U\res{R}\ind{G})_{N}\cong U_{N}\res{R/N}\ind{G/N}\cong U\res{R/N}\ind{G/N}.$$
	Hence $U\ds U\downarrow_{R}\uparrow^{G}$ as a $k\db{G}$-module. Thus $U$ is relatively $R$-projective, and hence ${}^g\!QN\leqslant R\leqslant QN$ for some $g\in G$ by Proposition \ref{teo:verticesareconjugate}, so that $R=QN$.
\end{proof}

\begin{prop}\label{prop:vtxopeninP}
	Let $U$ be an indecomposable finite dimensional $k\db{G}$-module with vertex $Q$. If $P$ is a $p$-Sylow subgroup of $G$ containing $Q$, then $Q$ is open in $P$. 
\end{prop}

\begin{proof}
	We must check that the index of $Q$ in $P$ is finite.   Since $U$ is finite dimensional, there is a cofinal system of open normal subgroups $N$ of $G$ acting trivially on $U$.  By Lemma \ref{lem:vtxsimplemodule}, $U$ has vertex $QN/N$ as a $k[G/N]$-module. Furthermore, by \cite[Proposition 2.2.3]{W} $PN/N$ is a Sylow $p$-subgroup of $G/N$ and so by \cite[Theorem 9 and Corollary 2]{Gr}, $|PN/N:QN/N|$ divides the dimension of $U$ for each $N$, and hence so does $|P:Q|$, showing that $Q$ is open in $P$.
\end{proof}

\section{Trace maps}\label{sectionTraceMaps}

Let $G$ be a profinite group and $A$ a pseudocompact $k$-algebra. We say that $A$ is a pseudocompact \emph{$G$-algebra} if $A$ is endowed with a continuous action $G\times A\to A$ of $G$ on $A$, written $(g,a)\mapsto {}^g\!a$ for all $a\in A$ and $g\in G$, such that the continuous map sending $a\in A$ to ${}^g\!a$ is a $k$-algebra automorphism. If $A,B$ are pseudocompact $G$-algebras, a continuous $k$-algebra homomorphism $f:A\to B$ is a \emph{homomorphism of $G$-algebras} if $f({}^g\!a)={}^g\!f(a)$ for all $a\in A$ and $g\in G$.

If $H$ is a closed subgroup of $G$, the \emph{subalgebra of $H$-fixed points} of $A$ is defined by 
$$A^{H}:=\{a\in A\,|\, {}^{h}a=a\,\forall h\in H\}.$$

\begin{defi}\label{def:tracemap}
	Let $A$ be a pseudocompact $G$-algebra and $H$ an open subgroup of $G$. The corresponding \emph{trace map} is the continuous linear map
	\begin{eqnarray*}
		\tn{Tr}_{H}^{G}:A^{H}&\to& A^{G}\\
		a&\mapsto& \sum_{g\in G/H}{}^g\!a,
	\end{eqnarray*}
	where $G/H$ denotes a set of left coset representatives of $H$ in $G$.
\end{defi}

\begin{rem}\label{rem:twomeaningsofTr}
	When $U$ is a finite dimensional $k\db{G}$-module, $\tn{End}_k(U)$ is a finite dimensional $G$-algebra with action $({}^g\!\rho)(u) := g\rho(g^{-1}u)$, and furthermore when $H$ is a closed subgroup of $G$ then $\tn{End}_k(U)^H = \tn{End}_{k\db{H}}(U)$.  It follows in particular that the notations $\tn{Tr}_H^G$ used in Definitions \ref{def:tracmap} and \ref{def:tracemap} are consistent for finite groups.  Formalizing the consistency for profinite groups is made awkward by the fact that $\tn{End}_k(U)$ is not pseudocompact when $U$ has infinite dimension, so to avoid unnecessary technicalities we will not attempt to do so.  The overlap of notation will not cause any difficulties.
\end{rem}

We list some basic properties of the trace map, whose proofs carry through just as for finite groups, see for instance \cite[Lemma 3.6.3]{B}.

\begin{lem}\label{lem:Tracemapproperties}
	Let $G$ be a profinite group and $A$ a pseudocompact $G$-algebra. Then
	\begin{enumerate}
		\item If $H$ is an open subgroup of $G$, for any $a\in A^{H}$ and $b\in A^{G}$ we have $b\,\tn{Tr}_{H}^{G}(a)=\tn{Tr}_{H}^{G}(ba)$ and $\tn{Tr}_{H}^{G}(a)b=\tn{Tr}_{H}^{G}(ab)$.
		\item If $H,L$ are open subgroups of $G$ with $L\leq H$, then $\tn{Tr}_{H}^{G}\circ \tn{Tr}_{L}^{H}=\tn{Tr}_{L}^{G}$.
		\item (Mackey's Formula) If $H,L$ are closed subgroups of $G$ with $L$ open, then for any $a\in A^{L}$ we have 
		$$\tn{Tr}_{L}^{G}(a)=\sum\limits_{g\in H\backslash G/L}\tn{Tr}_{H\cap {}^g\!L}^{H}({}^g\!a),$$
		where $H\backslash G/L$ denotes a set of double coset representatives.
	\end{enumerate}
\end{lem}

Definition \ref{def:tracemap} does not make sense for subgroups of $G$ that are not open.  We present two alternatives.  Firstly:

\begin{defi}\label{def:traceset}
	Let $A$ be a pseudocompact $G$-algebra. If $H$ is a closed subgroup of $G$, define $$\tn{oTr}_{H}^{G}(A^{H}) := \bigcap_{N\trianglelefteq_{O}G}\tn{Tr}_{HN}^{G}(A^{HN}).$$
\end{defi}

The letter ``o'' indicates ``open subgroups'' and is only included to minimize notational confusion.  Observe that, while $\tn{oTr}_H^G(A^H)$ is a well-defined subset of $A^G$, one cannot really regard $\tn{oTr}_H^G$ as a map.  In order to obtain a map, for simplicity we restrict to the case where $A = k\db{G}$, a $G$-algebra with action from $G$ given by conjugation, so for $g\in G$ and $x\in k\db{G}$ we have ${}^g\!x := gxg^{-1}$.  The reason to restrict is that $k\db{G}$ gives us access to the algebra homomorphisms $\varphi_{MN}, \varphi_N$ of Remark \ref{obs:invsystemforcoinvariants}, and because we will require only this special case.  The definition relies on Sylow's Theorem for profinite groups, for which see for instance \cite[Corollary 2.3.6]{ZR}.  We use the standard trick of realizing an inverse limit $A = \invlim_I \{A_i,\varphi_{ij}\}$ as
$$A = \bigg\{(a_i)_{i\in I}\in \prod_I A_i\,|\,\varphi_{ij}(a_j) = a_i\,\forall i\leqslant j\bigg\}.$$

If ever $n$ is a natural number, write $n=n_p\cdot n_{p'}$ where $n_p$ is a power of $p$ and $p$ does not divide $n_{p'}$.  Recall \cite[\S 2.3]{ZR} that if ever $H\leqslant L$ are closed subgroups of the profinite group $G$, the index of $H$ in $L$ is a well-defined supernatural number, with a well-defined decomposition 
$$|L:H| = |L:H|_p\cdot |L:H|_{p'}.$$

\begin{defi}\label{def:modifiedtrace}
	\begin{enumerate}
		\item Let $G$ be a finite group and $H\leqslant L$ subgroups of $G$.  The modified trace function $\widetilde{\tn{Tr}}_H^L : kG^H\to kG^L$ is defined as follows:
		$$\widetilde{\tn{Tr}}_H^L(x) := \frac{1}{|L:H|_{p'}}\tn{Tr}_{H}^{L}(x).$$
		
		\item Let $G$ be a profinite group and $H\leqslant L$ closed subgroups of $G$.  Suppose that $|L:H|_{p}$ is a finite number. The modified trace function $\widetilde{\tn{Tr}}_H^L : k\db{G}^H\to k\db{G}^L$ is defined as follows:
		$$\widetilde{\tn{Tr}}_H^L(x) := \left(\widetilde{\tn{Tr}}_{HN/N}^{LN/N}(\varphi_N(x))\right)_{N\in \mathcal{N}} =  \left(\frac{1}{|LN:HN|_{p'}}\tn{Tr}_{HN/N}^{LN/N}(\varphi_N(x))\right)_{N\in \mathcal{N}},$$
		where $\mathcal{N}$ is the set of open normal subgroups of $G$ such that $|LN:HN|_p = |L:H|_p$.   
	\end{enumerate}
\end{defi}

Some observations are called for:
\begin{itemize}
	\item The set $\mathcal{N}$ is non-empty (and hence cofinal in the set of all open normal subgroups of $G$) because $|L:H|_{p}$ is finite.  We do not define a trace function between subgroups without this property. In practice the subgroups $H$ to which we will apply $\widetilde{\tn{Tr}}_H^L$ are defect groups of a block of $G$, which we will see in Proposition \ref{prop:defgrpopen} are open subgroups of a $p$-Sylow subgroup of $G$.
	
	\item We prove only those properties of $\widetilde{\tn{Tr}}_H^L$ that we require to show Brauer's first main theorem in Section \ref{SectionBrauerCorr}.  There may be some value to a more systematic study of the linear map $\widetilde{\tn{Tr}}_H^L$.
\end{itemize}

We must check that the element $\widetilde{\tn{Tr}}_H^L(x)$ defined in the second part of the above definition is a well-defined element of $k\db{G}^L$.  This is a formal consequence of the first part of the following lemma (it is to make the vector compatible that we restrict our $N$ to the cofinal set $\mathcal{N}$).

\begin{lem}\label{Lemma tildeTr properties}
	Let $H\leqslant L$ be closed subgroups of $G$ with $|L:H|_p$ finite. 
	\begin{enumerate}
		\item Given $N\leqslant M$ in $\mathcal{N}$, we have $$\varphi_{MN}\widetilde{\tn{Tr}}_{HN/N}^{LN/N} = \widetilde{\tn{Tr}}_{HM/M}^{LM/M}\varphi_{MN}.$$
		\item If $|G:H|_{p}$ is finite, then
		$$\widetilde{\tn{Tr}}_L^G\widetilde{\tn{Tr}}_H^L = \widetilde{\tn{Tr}}_H^G.$$
		\item If $a\in k\db{G}^H$ and $b\in k\db{G}^L$ then $b\,\widetilde{\tn{Tr}}_{H}^{L}(a)=\widetilde{\tn{Tr}}_{H}^{L}(ba)$ and $\widetilde{\tn{Tr}}_{H}^{L}(a)b=\widetilde{\tn{Tr}}_{H}^{L}(ab)$.
	\end{enumerate}
\end{lem}

\begin{proof}
	\begin{enumerate}
		\item Note first that $(HM)(LN) = LM$ and so, by the second isomorphism theorem, 
		$$|LN:HM\cap LN| = |LM:HM|.$$  
		We have
		$$|LN:HN| = |LN:HM\cap LN|\cdot|HM\cap LN: HN| = |LM:HM|\cdot|HM\cap LN: HN|$$
		and hence, since $|LN:HN|_p = |LM:HM|_p$, we have
		$$|LN:HN|_{p'} = |LM:HM|_{p'}\cdot|HM\cap LN: HN|.$$
		
		Now
		\begin{align*}
			\varphi_{MN}\widetilde{\tn{Tr}}_{HN/N}^{LN/N}(x)
			& = \frac{1}{|LN:HN|_{p'}}\sum_{g\in LN/HN}{}^{gM}\varphi_{MN}(x) \\
			& = \frac{|HM\cap LN:HN|}{|LN:HN|_{p'}}\cdot\tn{Tr}_{HM/M}^{LM/M}\varphi_{MN}(x) \\ 
			& = \frac{1}{|LM:HM|_{p'}}\cdot\tn{Tr}_{HM/M}^{LM/M}\varphi_{MN}(x) \\
			& = \widetilde{\tn{Tr}}_{HM/M}^{LM/M}\varphi_{MN}(x)
		\end{align*}
		as required.
		
		\item For finite groups this is an easy check using the transitivity of the normal trace function (cf.\ Lemma \ref{lem:Tracemapproperties}).  The general case follows from this and Part 1.
		
		\item This also follows from the corresponding property for finite groups.
	\end{enumerate}
\end{proof}

\begin{lem}\label{lem:invlimoffintrace}
	Let $G$ be a profinite group and $D$ a closed subgroup of $G$ such that $|G:D|_p$ is finite.  Then
	$\tn{oTr}_{D}^{G}(k\db{G}^{D})\subseteq \tn{Im}(\widetilde{\tn{Tr}}_D^G)$.
\end{lem}

\begin{proof}
	It is sufficient to check that $\varphi_N(y)$ is in $\tn{Tr}_{DN/N}^{G/N}(k[G/N]^{DN/N})$ for any $y$ in $\tn{oTr}_{D}^{G}(k\db{G}^{D})$ and open normal subgroup $N$ of $G$.  But by definition $y = \sum_{g\in G/DN}{}^g\!a$ for some $a\in k\db{G}^{DN}$ and so
	$$\varphi_{N}(y) = \varphi_{N}\left(\sum_{g\in G/DN}{}^g\!a\right) = \tn{Tr}_{DN/N}^{G/N}(\varphi_N(a))\in \tn{Tr}_{DN/N}^{G/N}(k[G/N]^{DN/N}),$$
	as required.
\end{proof}

\section{Blocks of profinite groups}\label{sectionBlocks}

Throughout this section, $k$ is a discrete field of characteristic $p$ and $A$ is a pseudocompact $k$-algebra. An element $e\in A$ is \emph{idempotent} if $e^{2}=e$. Two idempotents $e,f$ of $A$ are \emph{orthogonal} if $ef=fe=0$. A non-zero idempotent is \emph{primitive} if it cannot be written as the sum of two non-zero orthogonal idempotents. We denote by $\tn{Z}(A)$ the center of $A$. An idempotent $e\in A$ is \emph{centrally primitive} if $e$ is a primitive idempotent of $\tn{Z}(A)$.

By \cite[IV. \S 3, Corollary 1,2]{G} applied to $\tn{Z}(A)$, there is a unique set of pairwise ortogonal centrally primitive idempotents $E=\{e_{i}\mbox{ : }i\in I\}$ in $A$ such that
\begin{eqnarray}\label{eq:blockdecomp}
	A=
	\prod\limits_{i\in I}Ae_{i}.
\end{eqnarray}
We call $E$ the complete set of centrally primitive orthogonal idempotents of $A$.  Any subset $F$ of $E$ is \emph{summable}, in the sense that each open ideal $I$ of $A$ contains all but finitely many of the elements of $F$, so that $\sum_{f\in F}f$ is a well-defined element of $A$.  We have $\sum_{e\in E}e = 1_A$.

Each $Ae_i$ is a pseudocompact algebra (with identity $e_i$), called a \emph{block} of $A$. The elements of $E$ are called \emph{block idempotents}.  The decomposition (\ref{eq:blockdecomp}) is called the \emph{block decomposition} of $A$.  Note that the block decomposition of $A$ is unique.

In what follows, $\overline{Y}$ denotes the topological closure of the subset $Y$ of $A$.

\begin{lem}\label{lem:idealsofeAe}
	Let $A$ be a pseudocompact algebra and let $e\in A$ be an idempotent. If $X$ is a not necessarily closed ideal of $A$, then $e\overline{X}e=\overline{eXe}$.
\end{lem}

\begin{proof}
	The map $\alpha:A\to eAe$ sending $a$ to $eae$ is continuous so that by Lemma \ref{JK-lem}, $e\overline{X}e$ is a closed subset of $eAe$ containing $eXe$, and hence $\overline{eXe}\subseteq e\overline{X}e$.  The other inclusion is direct from the continuity of $\alpha$.
\end{proof}

\begin{lem}\label{lem:Roslem}
	Let $A$ be a pseudocompact $k$-algebra and $e$ a primitive idempotent of $A$. Then
	\begin{enumerate}
		\item $eAe$ is a local algebra.
		\item (Rosenberg's Lemma) If $\mathcal{J}$ is a set of closed ideals of $A$ and $e\in \overline{\sum_{I\in\mathcal{J}}I}$, then there is some $I\in\mathcal{J}$ such that $e\in I$.
	\end{enumerate}
\end{lem}

\begin{proof}
	\begin{enumerate}
		\item This follows as in \cite[Lemma 1.3.3]{B}, using Lemma \ref{prop:localEndring}.
		
		\item By Lemma \ref{lem:idealsofeAe}, $e(\overline{\sum_{I\in\mathcal{J}}I})e=\overline{e(\sum_{I\in\mathcal{J}}I)e}=\overline{\sum_{I\in\mathcal{J}}eIe}$, so that if $e\in\overline{\sum_{I\in\mathcal{J}}I}$ then $e\in\overline{\sum_{I\in\mathcal{J}}eIe}$. Since the algebra $eAe$ is local, each ideal $eIe$, with $I\in\mathcal{J}$, is either contained in the Jacobson radical of $eAe$ or is equal to $eAe$. Thus there is at least one $I\in \mathcal{J}$ such that $eIe=eAe$, and $e$ is contained in this $I$.
	\end{enumerate}
\end{proof}

We say that the pseudocompact $A$-module $U$ \emph{lies} in the block $B$ if $BU=U$ and $B'U=0$ for any block $B'$ different from $B$. If $B$ has unity $e$, then $U$ lies in $B$ if, and only if, $eU=U$ and $fU=0$ for each centrally primitive idempotent $f$ distinct from $e$.  If $U$ lies in $B$ and $V$ is a closed submodule of $U$, then $V$ and $U/V$ lie in $B$. If $U_{1},U_{2}$ are $A$-modules lying in $B$ then $U_{1}\oplus U_{2}$ lies in $B$.

\begin{prop}\label{prop:moddecomp}
	Let $A = \prod B_i$ be the block decomposition of $A$ and let $U$ be a pseudocompact $A$-module. Then $U$ has a unique decomposition of the form $\prod\limits_{i\in I}U_{i},$ where $U_{i}$ is a submodule of $U$ lying in the block $B_{i}$.
\end{prop}

\begin{proof}
	Let $e_{i}$ be the block idempotent of $B_{i}$. Define $U_{i}=e_{i}U$, an $A$-module lying in $B_{i}$.  By Lemmas \ref{JK-lem} and \ref{lem:propertiesofLCspaces}, $\prod\limits_{i\in I}U_{i}$ is a pseudocompact $A$-module.

	Define for each $i\in I$ the continuous homomorphism $\rho_i: U\to U_i$ by $u\mapsto e_iu$, and for each finite subset $K$ of $I$, the continuous homomorphism $\rho_{K}:U\to \prod\limits_{i\in K}U_{i}$ by $\rho_{K}(u)=(\rho_{i}(u))_{i\in K}$.
	The homomorphisms $\rho_{K}$ induce a continuous surjective homomorphism $\rho:U\to \prod\limits_{i\in I}U_{i}$, which is injective since $\sum_{e\in E}e = 1_A$.  As continuous linear bijections between pseudocompact algebras are isomorphisms by Lemma \ref{lem:propertiesofLCspaces}, $U\cong \prod\limits_{i\in I}U_{i}$.  The uniqueness of the decomposition follows from the uniqueness of the complete set of centrally primitive orthogonal idempotents.
\end{proof}

In particular, indecomposable pseudocompact modules lie in a unique block.

\subsection{Finite dimensional modules and blocks}

Recall our convention that the symbols $\varphi_N$ and  $\varphi_{MN}$ are reserved for the canonical projections $\varphi_N : k\db{G}\to k[G/N], \varphi_{MN}: k[G/N]\to k[G/M]$.

\begin{rem}\label{obs:directsystemofcentralprimidemp}
	By \cite[Proposition 6.4]{PJ}, the complete set of centrally primitive orthogonal idempotents $E$ of $k\db{G}$ is a discrete set and can be obtained as the direct limit of the corresponding complete sets $E_N$ for the $k[G/N]$.  We make the direct system explicit, as we will use it in future: if $N\leq M$ are open normal subgroups of $G$, define $\psi_{MN}:E_{M}\to E_{N}$ by sending $c\in E_M$ to the unique centrally primitive idempotent $d$ of $k\left[G/N\right]$ such that $\varphi_{MN}(d)c\neq 0$.  Note that as $c$ is primitive as a central idempotent, this is equivalent to saying that $\varphi_{MN}(d)c = c$.
\end{rem}

\begin{lem}\label{lem:simplemoduleslyinginthesameblockofG/N}
	Let $U,V$ be finite dimensional $k\db{G}$-modules lying in the same block of $G$. Then there is $N_{0}\trianglelefteq_{O}G$ acting trivially on $U$, $V$ and such that $U,V$ lie in the same block of $G/N_{0}$. 
\end{lem}

\begin{proof}
	Since $U,V$ are finite dimensional, it is sufficient to prove the lemma supposing that both are indecomposable.  Let $e$ be the block idempotent such that $eU = U, eV = V$ and let $M$ be an open normal subgroup of $G$ acting trivially on both $U$ and $V$, so that both can be treated as $k[G/M]$-modules.  Note that they are indecomposable as such.
	
	We use the notation from Remark \ref{obs:directsystemofcentralprimidemp}.  If $f,g$ are the block idempotents of $U,V$ respectively in $k[G/M]$ then, since $\varphi_M(e)f = f$ and $\varphi_M(e)g = g$, we have that $\psi_M(f) = \psi_M(g) = e$.  It follows from the construction of the direct limit that there is an open normal subgroup $N_0$ of $G$ contained in $M$ such that $\psi_{MN_0}(f) = \psi_{MN_0}(g)$.  Both $U$ and $V$ are in the block of $k[G/N_0]$ with block idempotent $\psi_{MN_0}(f)$.
\end{proof}

\subsection{Inverse systems of blocks}

Consider $k\db{G}$ as a module for the completed group algebra $k\db{G\times G}$ with the following continuous multiplication:

\begin{eqnarray*}
	k\db{G\times G} \times k\db{G} &\to & k\db{G} \\
	((g_{1},g_{2}),x)&\mapsto &g_{1}xg_{2}^{-1},
\end{eqnarray*}
for $g_{1},g_{2}\in G$ and $x\in k\db{G}$.  The canonical projections $\varphi_{MN}: k[G/N]\to k[G/M], \varphi_N:k\db{G}\to k[G/N]$ are $k\db{G\times G}$-module homomorphisms.

The blocks of $G$ are precisely the indecomposable summands of $k\db{G}$ as a $k\db{G\times G}$-module and they are pairwise non-isomorphic as $k\db{G\times G}$-modules, since their annihilators as $k\db{G\times 1}$-modules are pairwise non-equal.

Let $B$ be a block of $G$ treated as a $k\db{G\times G}$-module, and $N\trianglelefteq_{O}G$. Define $B_{N}$ to be the coinvariant module $B_{N\times N}$, but note that $B_{N\times N}=B_{N\times 1}$, since given $m,n\in N$ and $x\in B_{N\times 1}$ there is $n'\in N$ such that
$$(m,n)x=mxn^{-1}=mn'x=(mn',1)x=x,$$
so that $N\times N$ acts trivially on $B_{N\times 1}$.
By Proposition \ref{prop:invlimcoinv}, $B=\varprojlim_{N}\{B_{N}, \varphi_{MN}\}$ where, as ever, $\varphi_{MN}$ is the canonical quotient map.

Denote by $\delta:G\to G\times G$ the diagonal map $g\mapsto (g,g)$.

\begin{lem}\label{lem:isomorphismofk[[G]]asbimodule}
	The $k\db{G \times G}$-module $k\db{G}$ is isomorphic to the induced module $k\uparrow_{\delta(G)}^{G\times G}$. 
\end{lem}

\begin{proof}
	The $k\db{G\times G}$-module homomorphisms $\rho:k\db{G}\to k\uparrow_{\delta(G)}^{G\times G}$ defined by $g\mapsto (g,1)\widehat{\otimes}1$, and $\gamma:k\uparrow_{\delta(G)}^{G\times G}\to k\db{G}$ defined by $(g,h)\widehat{\otimes}\lambda\mapsto g\lambda h^{-1}$, are continuous.  The checks that they are well-defined and mutually inverse are as in \cite[Proposition 5.11.6]{Lin1}.
\end{proof}

\begin{prop}\label{prop:invlimblocks}
	Let $B$ be a block of $G$.  Then $B$ is the inverse limit of a surjective inverse system of blocks of $G/N$ and algebra homomorphisms, for some cofinal set of open normal subgroups $N$ of $G$.
\end{prop}

\begin{proof}
	Let $E$ be the set of block idempotents of $k\db{G}$. We use the direct system $\{E_{N},\psi_{MN}\}$ described in Remark \ref{obs:directsystemofcentralprimidemp}. If $e\in E$ is the block idempotent of $B$, there is $N_{0}\trianglelefteq_{O}G$ and $e_{0}\in E_{N_{0}}$ such that $e=\psi_{N_{0}}(e_{0})$. Consider the cofinal system $\mathcal{N}$ of open normal subgroups $N$ of $G$ with $N\leq N_{0}$.

	For each $N\leq N_{0}$, let $e_{N}=\psi_{N_{0}N}(e_{0})$, the unique centrally primitive idempotent of $k[G/N]$ such that $\varphi_{N_{0}N}(e_{N})e_{0} = e_0$.  Observe that if ever $N\leqslant M\leqslant N_0$, then $\varphi_{MN}(e_N)e_M = e_M$. 
	The block $X_{N}$ of $G/N$ with block idempotent $e_{N}$ is a direct summand of $B_{N}$ as a $k\db{G\times G}$-module, so there are $k\db{G\times G}$-module homomorphisms $\iota_{N}:X_{N}\to B_{N}$ inclusion, and $\pi_{N}:B_{N}\to X_{N}$ multiplication by $e_{N}$, such that $\pi_{N}\iota_{N}=\tn{id}_{X_{N}}$. Now we can form a new inverse system of blocks $X_{N}$ with maps $\gamma_{MN}: = \pi_{M}\varphi_{MN}\iota_{N}$.  Note that $\gamma_{MN}(e_N) = \varphi_{MN}(e_N)e_M = e_M$, which implies in particular that the $\gamma_{MN}$ are surjective algebra homomorphisms. 
	Then $\{X_N,\gamma_{MN}\}$ forms an inverse system, because for all $N\leqslant M\leqslant L$ and all $x = xe_N\in X_N$ we have
	$$\gamma_{LM}\gamma_{MN}(x) = \pi_L\varphi_{LM}(\varphi_{MN}(x)e_M) = \varphi_{LN}(x)\varphi_{LM}(e_M)e_L = \gamma_{LN}(x).$$
	Denote by $X$ the inverse limit of this inverse system.  
	
	The split homomorphisms $\pi_{N}:B_N\to X_N$ are the components of a map of inverse systems, so by Lemma \ref{lem:splitmorphisms}, the induced continuous map $\pi=\varprojlim_{N}\pi_{N}$ is a split $k\db{G\times G}$-module homomorphism. Hence $X$ is a direct summand of $B$. But $B$ is indecomposable as a $k\db{G\times G}$-module and $X$ is not $0$, so
	$$B=X=\varprojlim_{N}X_{N}.$$
\end{proof}

\begin{rem}\label{obs:Basinvlimoffiniteblocks}
	Throughout this article, given a block $B$ of $G$ with block idempotent $e$, the notation $\{X_N,\gamma_{MN},N\in \mathcal{N}\}$ will refer to a fixed inverse system of blocks as constructed in Proposition \ref{prop:invlimblocks}.  This inverse system is not unique, due to the choice of the open normal subgroup $N_0$ and of the block idempotent $e_0$ of $G/N_0$, but it will be important that once these choices are made, the maps $\gamma_{MN}$ are canonical.
\end{rem}

\begin{cor}\label{cor:simplemodlies}
	Let $B$ be a block of $G$ with block idempotent $e$. Let $S$ be a finite dimensional $k\db{G}$-module lying in $B$. Then there is an open normal subgroup $N'$ of $G$ acting trivially on $S$ such that $S$ lies in the block $X_{N'}$ of $G/N'$. Furthermore, $S$ lies in $X_{N}$ for every $N \leqslant N'$. 
\end{cor}

\begin{proof}
	Write $B=\varprojlim_{N}\{X_{N},\gamma_{MN}, N\in \mathcal{N}\}$ 
	as above and fix $M\in \mathcal{N}$. There is a simple $k[G/M]$-module $T$ lying in $X_{M}$. Then $T$ lies in $B$. 
	If $N\leqslant M$ then $e_NT=T$, since otherwise $e_NT = 0$, which would imply that $e_MT = e_M\varphi_{MN}(e_N)T = e_M\cdot e_NT = 0$. So $T$, treated as a $k[G/N]$-module, lies in $X_{N}$.
	
	Since $S,T$ are finite dimensional $k\db{G}$-modules lying in the same block $B$, by Lemma \ref{lem:simplemoduleslyinginthesameblockofG/N} there is $N'\trianglelefteq_{O}G$ with $N'\leq M$, such that $S,T$ lie in the same $k[G/N']$-block. But $T$ lies in $X_{N'}$, hence so does $S$.
\end{proof}

\begin{prop}\label{prop:relprojforblocks}
	Let $G$ be a profinite group, $R$ a closed subgroup of $G$ and $B$ a block of $G$ treated as a $k\db{G\times G}$-module. Then $B$ is $\delta(R)$-projective if, and only if, $X_{N}$ is $\delta(R)(N\times N)$-projective for each $N\trianglelefteq_{O}G$. 
\end{prop}

\begin{proof}
	The forward implication is easy: if $B$ is relatively $\delta(R)$-projective then $B_N = B_{N\times N}$ is relatively $\delta(R)(N\times N)$-projective.  But $X_N$ is a direct summand of $B_N$, so we are done.
	
	For the reverse implication, suppose that $X_{N}$ is $\delta(R)(N\times N)$-projective for each $N\trianglelefteq_{O}G$.
	
	We utilize the natural isomorphism of Lemma \ref{lem:isoind/res} to write
	$$B_N\res{\delta(R)(N\times N)}\ind{G\times G} = k[G\times G / \delta(R)(N\times N)]\ctens_k B_N.$$
	Letting $N$ vary and setting $\Psi_{MN}$ to be the tensor product of the natural projections, we obtain an inverse system of $k\db{G\times G}$-modules $\{B_N\res{\delta(R)(N\times N)}\ind{G\times G},\,\Psi_{MN}\}$, whose inverse limit is the module {$B\res{\delta(R)}\ind{G\times G}$} by \cite[Proposition 5.2.2]{ZR}.  Denoting by $\pi_N : B_N \to X_N$ the canonical projection, we define the continuous $k\db{G\times G}$-module homomorphism
	\begin{align*}
		q_N : k[G\times G / \delta(R)(N\times N)]\ctens_k B_N & \to X_N \\
		z\ctens b & \mapsto \pi_N(b).
	\end{align*}
	Note that $q_N$ splits as a $k\db{G\times G}$-module homomorphism, being the composition of split homomorphisms $B_N\res{\delta(R)(N\times N)}\ind{G\times G}\to X_N\res{\delta(R)(N\times N)}\ind{G\times G}$ and $X_N\res{\delta(R)(N\times N)}\ind{G\times G}\to X_N$ (the latter split since $X_N$ is relatively $\delta(R)(N\times N)$-projective).  Furthermore the $q_N$ yield a map of inverse systems, as may be easily checked.  The induced homomorphism $q : B\res{\delta(R)}\ind{G\times G}\to \invlim X_N = B$ is thus split by Lemma \ref{lem:splitmorphisms}, so that $B$ is relatively $\delta(R)$-projective by Theorem \ref{teo:projrel}.
\end{proof}

\section{Defect Groups}\label{section:defectgroups}

\subsection{Definition and basic properties}

Recall that we treat $k\db{G}$ as a $G$-algebra with action from $G$ given by conjugation.  Whenever we treat a subalgebra or a quotient algebra of $k\db{G}$ as a $G$-algebra, the action from $G$ is induced from this one.  Note that the center $\tn{Z}(k\db{G})$ of $k\db{G}$ is $k\db{G}^{G}$, so block idempotents of $k\db{G}$ belong to $k\db{G}^{G}$.  Recall the definition \ref{def:traceset} of $\tn{oTr}_D^G(k\db{G}^D)$.

\begin{defi}\label{def:defgrp1}
	Let $B$ be a block of a profinite group $G$ with block idempotent $e$. A \emph{defect group} of $B$ is a closed subgroup $D$ of $G$ such that $e\in \tn{oTr}_{D}^{G}(k\db{G}^{D})$ and minimal with this property.
\end{defi}

\begin{teo}\label{teo:defgrpG-algfin}
	Let $B$ be a block of $G$.  A defect group of $B$ exists.
\end{teo}   

\begin{proof}
	Let $e$ be the block idempotent of $B$ and consider $\mathcal{D}=\{H\leq_{C}G\,:\,e\in \tn{oTr}_{H}^{G}(k\db{G}^{H})\}$. Then $\mathcal{D}$ is partially ordered by inclusion and non-empty, since $G\in \mathcal{D}$.  If $\mathcal{C}$ is a chain in $\mathcal{D}$, we assert that $L=\bigcap\limits_{H\in\mathcal{C}}H\in\mathcal{D}$.  Fix $N\trianglelefteq_{O}G$.  As $LN$ is open in $G$, $LN=HN$ for some $H\in \mathcal{C}$ and hence $e\in \tn{Tr}_{LN}^{G}(k\db{G}^{LN})$. Since $N$ was arbitrary, $e\in \tn{oTr}_L^G(k\db{G}^L)$ by definition and hence $L\in \mathcal{D}$ as claimed.  The existence of a defect group for $B$ now follows from Zorn's Lemma.
\end{proof}

In several proofs that follow we use $\tn{Tr}$ to mean two different things (the $\tn{Tr}$ of Definitions \ref{def:tracmap} and \ref{def:tracemap}), but see Remark \ref{rem:twomeaningsofTr}.  

\begin{prop}\label{prop:modrelprojtodefgrp}
	If $B$ is a block of $G$ with defect group $D$, then any finitely generated $k\db{G}$-module lying in $B$ is relatively $D$-projective.
\end{prop}

\begin{proof}
	Let $U$ be a finitely generated $k\db{G}$-module lying in $B$.
	For each $N\trianglelefteq_{O}G$ we have $e\in \tn{Tr}_{DN}^{G}(k\db{G}^{DN})$ by hypothesis, so there is $a_{N}\in k\db{G}^{DN}$ such that  $e= \tn{Tr}_{DN}^{G}(a_{N})$. Denote by $\alpha_{N}:U\to U$ the continuous $k\db{DN}$-module endomorphism defined by $u\mapsto a_{N}u$.  
	One easily checks that $\tn{id}_{U}=\tn{Tr}_{DN}^{G}(\alpha_{N})$, which implies by Theorem \ref{teo:projrel} that $U$ is relatively $D$-projective. 
\end{proof}

\begin{prop}\label{prop:defgrpchar}
	Let $B$ be a block of $G$.  Then $B$ has a vertex of the form $\delta (D)$ as a $k\db{G\times G}$-module, where $D$ is a pro-$p$ subgroup of $G$.
\end{prop}

\begin{proof}
	By Lemma \ref{lem:isomorphismofk[[G]]asbimodule}, $B$ is relatively $\delta(G)$-projective.  The result now follows since vertices are pro-$p$ subgroups of $G$ by Lemma \ref{lem:vtxpro-psubgrp}.
\end{proof}

\begin{prop}\label{prop:equivrelproj}
	Let $H$ be a closed subgroup of $G$ and $B$ a block of $G$ with block idempotent $e$. Then $e\in \tn{oTr}_{H}^{G}(k\db{G}^{H})$ if, and only if, $B$ is $\delta(H)$-projective as a $k\db{G\times G}$-module.
\end{prop}

\begin{proof}
	Suppose that  $e\in \tn{oTr}_{H}^{G}(k\db{G}^{H})$ so that for each $N\trianglelefteq_{O}G$, there is $a_{N}\in k\db{G}^{HN}$ such that $e=\tn{Tr}_{HN}^{G}(a_{N})$. We claim first that $B\downarrow_{\delta(G)(N\times N)}$ is $\delta(H)(N\times N)$-projective. Consider the continuous $k\db{\delta(H)(N\times N)}$-module homomorphism $\alpha_{N}:B\to B$ given by $y\mapsto a_{N}y$.  Observe that if $R$ is a set of left coset representatives of $HN$ in $G$, then $\{(r,r)\,:\,r\in R\}$ is a set of left coset representatives of $\delta(H)(N\times N)$ in $\delta(G)(N\times N)$. So
	\begin{eqnarray*}
		\tn{Tr}_{\delta(H)(N\times N)}^{\delta(G)(N\times N)}(\alpha_{N})(e)&=&\sum_{(r,r)\in \delta(G)(N\times N)/\delta(H)(N\times N)}(r,r)\alpha_{N}((r^{-1},r^{-1})e)\\
		&=&\sum_{(r,r)\in \delta(G)(N\times N)/\delta(H)(N\times N)}r\alpha_{N}(r^{-1}er))r^{-1}\\
		&=&\sum_{(r,r)\in \delta(G)(N\times N)/\delta(H)(N\times N)}ra_{N}r^{-1}err^{-1}\\
		&=&\tn{Tr}_{HN}^{G}(a_{N})e\\
		&=&e.
	\end{eqnarray*}
	
	Thus $\tn{Tr}_{\delta(H)(N\times N)}^{\delta(G)(N\times N)}(\alpha_{N})$ is the identity map on $B$ and so by Theorem \ref{teo:projrel}, $B\downarrow_{\delta(G)(N\times N)}$ is $\delta(H)(N\times N)$-projective.  Let $Z$ be a $k\db{\delta(H)(N\times N)}$-module such that
	$$ B\res{\delta(G)(N\times N)}\ds Z\ind{\delta(G)(N\times N)}.$$
	Then, since $B$ is $\delta(G)(N\times N)$-projective for each $N\trianglelefteq_{O}G$ by Lemma \ref{lem:isomorphismofk[[G]]asbimodule}, we have
	\begin{eqnarray*}
		B&\ds &B\downarrow_{\delta(G)(N\times N)}\uparrow^{G\times G}\ds Z\uparrow^{\delta(G)(N\times N)}\uparrow^{G\times G}
	\end{eqnarray*}
	which implies that $B$ is relatively $\delta(H)(N\times N)$-projective for each $N\trianglelefteq_{O}G$, and hence is relatively $\delta(H)$-projective by Theorem \ref{teo:projrel}.
	
	To show the reverse implication, suppose that $B$ is relatively $\delta(H)$-projective.  Then $B$ is relatively $(G\times HN)$-projective for each $N\trianglelefteq_{O}G$. By Theorem \ref{teo:projrel}, there is thus a continuous $k\db{G\times HN}$-module endomorphism $\alpha_{N}$ of $B$ such that $\tn{id}_{B}=\tn{Tr}_{G\times HN}^{G\times G}(\alpha_{N})$.  Arguing as in the finite case \cite[Proposition 6.2.3]{Lin2} it follows that $e\in \tn{Tr}_{HN}^{G}(k\db{G}^{HN})$, and hence $e\in\bigcap_{N}\tn{Tr}_{HN}^{G}(k\db{G}^{HN})=\tn{oTr}_{H}^{G}(k\db{G}^{H})$, as required.
\end{proof}

\begin{prop}\label{prop:equivdefgrp}
	Let $D$ be a closed subgroup of $G$ and let $B$ be a block of $G$. The following statements are equivalent.
	\begin{enumerate}
		\item $D$ is a defect group of $B$.
		\item $\delta(D)$ is a vertex of $B$ as a $k\db{G\times G}$-module.
	\end{enumerate}
\end{prop}

\begin{proof}
	This follows from Proposition \ref{prop:equivrelproj}.
\end{proof}

\begin{prop}\label{prop:defectgrpareconjclassofpro-pgrps}
	Let $G$ be a profinite group and $B$ a block of $G$. The defect groups of $B$ are a conjugacy class of pro-$p$ subgroups of $G$.
\end{prop}

\begin{proof}
	If $C,D$ are defect groups of $B$ then they are pro-$p$ subgroups of $G$ and $\delta(C), \delta(D)$ are conjugate in $G\times G$ by Proposition \ref{prop:equivdefgrp} and Proposition \ref{teo:verticesareconjugate}.   As with finite groups (cf.\ \cite[p.\ 98]{A}) a direct calculation shows that this implies that $C$ and $D$ are conjugate in $G$.
\end{proof}

\begin{prop}\label{prop:defgrpopen}
	Let $B$ be a block of $G$ with defect group $D$ and let $P$ be a $p$-Sylow subgroup of $G$ containing $D$. Then $D$ is open in $P$.
\end{prop}

\begin{proof}
	Let $S$ be a simple $k\db{G}$-module lying in $B$.  By Proposition \ref{prop:modrelprojtodefgrp}, $S$ has a vertex $Q$ contained in $D$.  But $Q$ is open in $P$ by Proposition \ref{prop:vtxopeninP}, hence so is $D$.
\end{proof}

\subsection{Defect groups and inverse systems}

We show that defect groups are well-behaved with respect to the inverse system of finite dimensional blocks of Remark \ref{obs:Basinvlimoffiniteblocks}.  The following fact, due to Green in the finite case \cite[Theorem 12]{Gr}, is rarely mentioned for finite groups, but becomes incredibly useful for profinite groups because simple modules are finite dimensional:

\begin{prop}\label{prop:simplemodvtxD}
	Let $G$ be a profinite group and $B$ a block of $G$ with defect group $D$. There is a simple module $T$ lying in $B$ with vertex $D$.
\end{prop}

\begin{proof}
	Recall the notation $X_N$ of Remark \ref{obs:Basinvlimoffiniteblocks}.  Let $T$ be a simple $k\db{G}$-module lying in $B$ with the following property: in an algebraic closure $\overline{k}$ of $k$, the module $\overline{k}\otimes_k T$ has an indecomposable direct summand of dimension $p^ra$ with $p$ not dividing $a$ and $r$ as small as possible among all simple $k\db{G}$-modules lying in $B$.  By Proposition \ref{prop:modrelprojtodefgrp}, $T$ has a vertex $Q$ inside $D$.  Using Corollary \ref{cor:simplemodlies}, fix an inverse system $\mathcal{N}$ of open normal subgroups $N$ of $G$ acting trivially on $T$ and such that $T$ lies in $X_N$.  
	
	Given $N\in \mathcal{N}$, $T$ has vertex $QN/N$ as a $k[G/N]$-module by Lemma \ref{lem:vtxsimplemodule}.
	As simple $k[G/N]$-modules are simple as $k\db{G}$-modules, one may follow the arguments of \cite[Theorem 2.2]{Knorr} and \cite[Theorem 12]{Gr} to see that our choice of $T$ implies that $X_N$ has defect group $QN/N$.

	By Proposition \ref{prop:equivdefgrp}, $X_N$ thus has vertex $\delta(QN/N)$ as a $k[G/N \times G/N]$-module.  We have a natural isomorphism of groups $\delta(\frac{QN}{N})\cong\frac{\delta(Q)(N\times N)}{N\times N}$ and hence $X_N$ is relatively $\delta(Q)(N\times N)$-projective as a $k\db{G\times G}$-module.  It now follows from Proposition \ref{prop:relprojforblocks} that $B$ is $\delta(Q)$-projective and hence that $D=Q$.
\end{proof}

\begin{cor}\label{cor:defgrpsG/N-blocksX_N}
	Let $B$ be a block of
	$G$ with defect group $D$. There is $N_{0}\trianglelefteq_{O}G$ such that $X_N$ has defect group $DN/N$ for every open normal subgroup $N$ of $G$ contained in $N_0$.
\end{cor}

\begin{proof}
	By Proposition \ref{prop:simplemodvtxD}, we may consider a simple $k\db{G}$-module $T$ lying in $B$ with vertex $D$.  Fix $N_0$ such that $T$ lies in $X_{N_0}$ and consider an open normal subgroup $N$ of $G$ contained in $N_0$.  By Lemma \ref{lem:vtxsimplemodule}, $T$ has vertex $DN/N$ as a $k[G/N]$-module and hence by Proposition \ref{prop:modrelprojtodefgrp}, there is a defect group $D'/N$ of $X_N$ containing $DN/N$.  But on the other hand, since $B$ is relatively $\delta(D)$-projective as a $k\db{G\times G}$-module, $X_N$ is relatively $\delta(DN/N)$-projective as a $k[G/N \times G/N]$-module, and so $DN/N = D'/N$, as required.
\end{proof}

We can now be more precise about the nature of defect groups:

\begin{prop}\label{prop:defectgroupsintersectionofsylows}
	Let $D$ be a defect group of the block $B$ of $G$ and let $P$ be a $p$-Sylow subgroup of $G$ containing $D$.  There is $g$ in the centralizer $\tn{C}_G(D)$ of $D$ in $G$ such that $D = P\cap {}^g\!P$.
\end{prop}

\begin{proof}
	By Corollary \ref{cor:defgrpsG/N-blocksX_N}, work within the cofinal system of $N$ for which $X_N$ has defect group $DN/N$.  By the finite version of this result \cite[Theorem 12.3.3]{Web} there is a non-empty closed set $C_N$ of $g\in \tn{C}_G(DN)$ such that $DN = PN\cap {}^g\!PN$.  By compactness, the intersection $C$ of the $C_N$ is non-empty and if $g\in C$ then $D = P\cap {}^g\!P$.
\end{proof}

\begin{rem}
	Propositions \ref{prop:defgrpopen} and \ref{prop:defectgroupsintersectionofsylows} together say that the candidates for a defect group may be quite limited.  For instance, if $G$ has the property that the intersection of a $p$-Sylow subgroup $P$ with any proper conjugate of $P$ has infinite index in $P$, or if $\textnormal{C}_G(D) \leqslant {N}_G(P)$ for every open subgroup $D$ of $P$, then every block of $G$ has defect group $P$.  
\end{rem}

\begin{cor}\label{cor:defectgrpmaximalinnormalizer}
	Let $G$ be a profinite group and $B$ a block of $G$ with defect group $D$. Then $D$ is the largest normal pro-$p$ subgroup of $\tn{N}_G(D)$. 
\end{cor}

\begin{proof}
	Given Proposition \ref{prop:defectgroupsintersectionofsylows}, the proof of this is just like for finite groups \cite[Corollary 12.3.4]{Web}.
\end{proof}

\begin{prop}\label{prop:blockidempininvlimoftraces}
	Let $G$ be a profinite group and $B$ a block of $G$  with block idempotent $e$. A closed subgroup $D$ of $G$ is a defect group of $B$ if, and only if, it is minimal among the closed subgroups $R$ of $G$ such that $|G:R|_p<\infty$ with the property that
	$$e\in \widetilde{\tn{Tr}}_{R}^{G}(k\db{G}^{R}).$$
\end{prop}

% \begin{prop}\label{prop:blockidempininvlimoftraces}
% 	Let $G$ be a profinite group and $B$ a block of $G$  with block idempotent $e$. A closed subgroup $D$ of $G$ is a defect group of $B$ if, and only if, it is minimal with the property that
% 	$$e\in \widetilde{\tn{Tr}}_{D}^{G}(k\db{G}^{D}).$$
% \end{prop}

\begin{proof}
%Recall that when $D$ is a defect group of $B$, then $|G:D|_p<\infty$.
When $D$ is a defect group of $B$, $|G:D|_p<\infty$ by Proposition \ref{prop:defgrpopen}, so by Lemma \ref{lem:invlimoffintrace}, $e\in \widetilde{\tn{Tr}}_{D}^{G}(k\db{G}^{D})$.  Let $R$ be an open subgroup of $D$ with the property that 
	$$e\in \widetilde{\tn{Tr}}_{R}^{G}(k\db{G}^{R}).$$
	We must check that $R=D$.  
	
	By Proposition \ref{prop:simplemodvtxD}, let $T$ be a simple $k\db{G}$-module lying in $B$ with vertex $D$.  We work within the cofinal inverse system of $N$ for which $T$ is in $X_N$  (as we may by Corollary \ref{cor:simplemodlies}).  Fix such an $N$.  By hypothesis, $\varphi_{N}(e)=\tn{Tr}_{RN/N}^{G/N}(a)$ for some $a$ in $k[G/N]^{RN/N}$.  The $k[RN/N]$-module endomorphism $\beta:T\to T$ sending $x$ to $ax$ is such that $\tn{id}_{T}=\tn{Tr}_{RN/N}^{G/N}(\beta_{N})$, and hence $T = T_N$ is relatively $RN/N$-projective.  Since $N$ was arbitrary, $T$ is thus relatively $R$-projective as a $k\db{G}$-module by Theorem \ref{teo:projrel}.  But $R$ is a subgroup of $D$ and $D$ is a vertex of $T$, so $R=D$ as required.
\end{proof}

\subsection{The Brauer homomorphism}

We give two further characterizations of defect groups.

Consider as ever the continuous action of $G$ on $k\db{G}$ by conjugation.  When $D$ is a closed subgroup of $G$ and $Q$ an open subgroup of $D$, then $\tn{Tr}_{Q}^{D}(k\db{G}^{Q})$ is an ideal of $k\db{G}^D$ by Part 1 of Lemma \ref{lem:Tracemapproperties}.  Thus $\sum\limits_{Q\lneqq_{O}D}\tn{Tr}_{Q}^{D}(k\db{G}^{Q})$ is an abstract ideal of $k\db{G}^{D}$ and so its closure is a closed ideal of $k\db{G}^D$ \cite[Theorem 4.2]{War}.

\begin{defi}\cite[Definition 3.2]{P2}\label{def:Brquotient}\label{def:Brhom}
	Let $G$ be a profinite group and $D$ a closed pro-$p$ subgroup of $G$.  The corresponding \emph{Brauer quotient} is the quotient algebra 
	
	$$      k\db{G}^{[D]} := k\db{G}^{D}/\overline{\sum\limits_{Q\lneqq_{O}D}\tn{Tr}_{Q}^{D}(k\db{G}^{Q})}.$$
	
	The \emph{Brauer homomorphism} is the natural projection
	$$\tn{Br}_{D}:k\db{G}^{D}\to k\db{G}^{[D]}.$$ 
\end{defi}

\begin{prop}\label{prop:defgrpcharbyBr_D}
	Let $B$ be a block of a profinite group $G$ with block idempotent $e$. Then $B$ has defect group $D$ if, and only if, $e \in \tn{oTr}_{D}^{G}(k\db{G}^{D})$ and $\tn{Br}_{D}(e)\neq 0$.
\end{prop}

\begin{proof}
	Suppose first that $B$ has defect group $D$.  Then $e \in \tn{oTr}_{D}^{G}(k\db{G}^{D})$ by definition, so we need only check that $\tn{Br}_{D}(e)\neq 0$.  If $\tn{Br}_D(e) = 0$, then by Rosenberg's Lemma \ref{lem:Roslem}, $e\in \tn{Tr}_{Q}^{D}(k\db{G}^{Q})$ for some proper open subgroup $Q$ of $D$.  Fix an open normal subgroup $M$ of $G$ such that $X_M$ has defect group $DM/M$ (by Corollary \ref{cor:defgrpsG/N-blocksX_N}) and such that $QM/M$ has the same index in $DM/M$ as $Q$ has in $D$.  Then $\varphi_M(e)$ is an element of $\tn{Tr}_{QM/M}^{DM/M}(k[G/M]^{QM/M})$ so that $\tn{Br}_{DM/M}(\varphi_M(e)) = 0$.  But if $f$ is the block idempotent of $X_M$ then $\varphi_M(e)f = f$, and hence $\tn{Br}_{DM/M}(f) = 0$ since $\tn{Br}_{DM/M}$ is an algebra homomorphism.  This contradicts the finite version of the result \cite[Lemma 6.2.4]{B}, since $X_M$ has defect group $DM/M$.
	
	Suppose now that $e\in \tn{oTr}_{D}^{G}(k\db{G}^{D})$ and $\tn{Br}_{D}(e)\neq 0$.  Then certainly $B$ has a defect group $R$ contained in $D$.  
	Note that $R$ is open in $D$ by Proposition \ref{prop:vtxopeninP}, so we can find an open normal subgroup $M$ of $G$ such that $D\cap RM = R$.  We have $e = \tn{Tr}_{RM}^G(a)$ for some $a\in k\db{G}^{RM}$ by hypothesis.  Applying Mackey's formula (Lemma \ref{lem:Tracemapproperties}),
	$$e = \sum_{g\in D\backslash G/RM} \tn{Tr}_{D\cap {}^g\!RM}^{D}({}^g\!a).$$
	We claim that if $R$ is a proper subgroup of $D$, then so is each of the $D\cap {}^g\!RM$: if for some $g$ we had $D\cap {}^g\!RM = D  $, then $R < D \leqslant {}^g\!RM$, which implies that $RM \leqslant {}^g\!RM$ and hence that $g$ is in the normalizer $\tn{N}_G(RM)$.  But then $D\cap {}^g\!RM = D\cap RM = R<D$, a contradiction.  It follows that if $R$ is proper in $D$, then $e$ is in the kernel of $\tn{Br}_D$, contrary to our hypothesis.  Hence $R=D$ and we are done.
\end{proof}

For the convenience of the reader we collect together the characterizations of a defect group proved above, and provide one further characterization that will help us to prove Brauer's first main theorem.  We need one more simple lemma.

\begin{lem}\label{lem:conjugatedsubgrpsviaBr_D}
	Let $R,D$ be closed pro-$p$ subgroups of $G$. If, for each $N\trianglelefteq_{O}G$, $\tn{Br}_{D}(\tn{Tr}_{RN}^{G}(a))\neq0$ for some $a\in k\db{G}^{RN}$, then $D$ is conjugate to a subgroup of $R$.
\end{lem}

\begin{proof}
	Fix $N\trianglelefteq_{O}G$ and $a$ such that $\tn{Br}_{D}(\tn{Tr}_{RN}^{G}(a))\neq 0$.  By the Mackey formula,
	$$0\neq \tn{Br}_D(\tn{Tr}_{RN}^G(a)) = \sum_{g\in D\backslash G/RM}  \tn{Br}_D(\tn{Tr}_{D\cap {}^g\!RN}^D({}^g\!a)),$$
	and hence the closed set $C_N$ of $g\in G$ with $D\subseteq {}^g\!RN$ is non-empty.  A standard compactness argument yields that $C = \bigcap_{N\trianglelefteq_{O}G}C_N$ is non-empty, and $g\in C$ has the property that $D\subseteq {}^g\!R$.
\end{proof}

\begin{teo}\label{theorem:CompleteDefectGroupChar}
	Let $G$ be a profinite group and $B$ a block of $G$ with block idempotent $e$.  The following are equivalent for a closed subgroup $D$ of $G$:
	\begin{enumerate}
		\item $e\in \tn{oTr}_D^G(k\db{G}^D)$ and $D$ is minimal with this property,
		\item $|G:D|_p<\infty$, $e\in \widetilde{\tn{Tr}}_D^G(k\db{G}^D)$ and $D$ is minimal among the closed subgroups of $G$ having these two properties,
		%$e\in \widetilde{\tn{Tr}}_D^G(k\db{G}^D)$ and $D$ is minimal among the closed subgroups $R$ of $G$ such that $|G:R|_p<\infty$ with this property,
		\item $\delta(D)$ is a vertex of $B$ as a $k\db{G\times G}$-module,
		\item $e\in \tn{oTr}_D^G(k\db{G}^D)$ and $\tn{Br}_D(e)\neq 0$.
		
		\item $\tn{Br}_D(e)\neq 0$ and $D$ is maximal among the pro-$p$ subgroups of $G$ with this property.
	\end{enumerate}
\end{teo}

\begin{proof}
	The equivalence of $1,2,3,4$ is the content of Propositions \ref{prop:equivdefgrp}, \ref{prop:defgrpcharbyBr_D} and \ref{prop:blockidempininvlimoftraces}.  To conclude the result, we will check that $4$ and $5$ are equivalent.
	
	Suppose that $D$ satisfies $4$ and let $R$ be a pro-$p$ subgroup of $G$ containing $D$ such that $\tn{Br}_R(e)\neq 0$.  Lemma \ref{lem:conjugatedsubgrpsviaBr_D} (applied with the roles of $R$ and $D$ interchanged) implies that $D\subseteq R\subseteq {}^g\!D$ for some $g\in G$ and hence $R=D$.  
	
	Suppose now that $D$ satisfies $5$.  Let $R$ be a defect group of $B$, so that $e\in \tn{oTr}_R^G(k\db{G}^R)$ and $\tn{Br}_R(e)\neq 0$.  By Lemma \ref{lem:conjugatedsubgrpsviaBr_D} (with $R$ and $D$ the right way round this time), $D$ is conjugate to a subgroup of $R$, and hence by Proposition \ref{prop:defectgrpareconjclassofpro-pgrps} we may as well choose $R$ containing $D$.  But $D$ is by hypothesis maximal with $\tn{Br}_D(e)\neq 0$ and hence $D=R$, as required.
\end{proof}

\section{Brauer's first main theorem for profinite groups}\label{SectionBrauerCorr}

Brauer's first main theorem for finite groups (cf. \cite[Theorem 6.2.6]{B}) is a natural bijection between the blocks of a finite group $G$ with defect group $D$, and the blocks of the normalizer of $D$ in $G$ with defect group $D$.  In this section we use the results thus obtained to demonstrate our main theorem, a version of Brauer's first main theorem for profinite groups.  We follow the approaches of \cite[\S 6.2]{B} and \cite[\S 12.6]{Web}.  Given the results already obtained, the proof is mostly similar to the finite case, with the exception of  Lemma \ref{lem:Brdiagramcommutes}.  Remark \ref{Remark:BrauerCorrFails} provides a simple example that illustrates a certain subtlety of the correspondence.

\begin{lem}\label{lem:centralizerisomorphictok[[G]]^[D]}
	Let $D$ be a closed pro-$p$ subgroup of $G$.  The composition $$k\db{\tn{C}_G(D)}\to k\db{G}^{D}\to k\db{G}^{[D]},$$ 
	where the first map is induced from the inclusion of the centralizer of $D$ in $G$ into $k\db{G}^{D}$ and the second is the Brauer homomorphism, is an isomorphism of $\tn{N}_G(D)$-algebras.
\end{lem}

\begin{proof}
	This is \cite[Lemma 3.4]{P2}.
\end{proof}

It will be convenient for us to consider $\tn{Br}_D$ as taking values in $k\db{C_G(D)}$.  Thus we will regard $\tn{Br}_D$ as the map
$$k\db{G}^{D}\to k\db{G}^{[D]}\xrightarrow{\sim} k\db{\tn{C}_{G}(D)},$$
where the second arrow is the inverse of the isomorphism of Lemma \ref{lem:centralizerisomorphictok[[G]]^[D]}.  

Lemma \ref{lem:Brdiagramcommutes} below corresponds to \cite[Lemma 6.2.5]{B}. It is the principal reason we introduce the modified trace function in Section \ref{sectionTraceMaps} and its proof is rather delicate. It seems to be a fundamental obstruction that the lemma cannot be formulated without the hypothesis that $D$ is open in $P$ -- this is reflected in the failure of the correspondence when $D$ has infinite index in $P$ (see Remark \ref{Remark:BrauerCorrFails}).  The following auxiliary lemma is not strictly necessary for the proof of Lemma \ref{lem:Brdiagramcommutes}, but without it the argument may look rather suspicious.

\begin{lem}\label{lemma index independent of N}
	Let $D$ be a pro-$p$ subgroup of $G$ that is open in a $p$-Sylow subgroup of $G$ that contains it. Let $\mathcal{N}$ denote the set of open normal subgroups of $G$ for which $D$ is a $p$-Sylow subgroup of $DN$.  The element $|G:\tn{N}_G(DN)|_{p'}$ of $k$ does not depend on $N\in \mathcal{N}$.
\end{lem}

\begin{proof}
	It is sufficient to prove that if $G$ is a finite group with $p$-subgroup $D$ and $N$ is a normal subgroup such that $D$ is a Sylow $p$-subgroup of $DN$, then $|G:\tn{N}_G(D)|_{p'} = |G/N:\tn{N}_{G/N}(DN/N)|_{p'}$ in $k$.  By the orbit-stabilizer theorem, $|G:\tn{N}_G(D)|_{p'}$ is the $p'$-part of the size of the set ${}^G\!D$ of $G$-conjugates of $D$ in $G$, while $|G/N:\tn{N}_{G/N}(DN/N)|_{p'}$ is the $p'$-part of the size of the set ${}^{G/N}\!DN/N$ of $G/N$-conjugates of $DN/N$ in $G/N$.  Consider the surjective map $\rho:{}^G\!D\to {}^{G/N}\!DN/N$ given by ${}^g\!D\mapsto {}^{gN}DN/N$.  By our hypothesis on $N$, the inverse image under $\rho$ of an element ${}^{gN}DN/N$ is precisely the set $\tn{Syl}_p({}^g\!DN)$ of Sylow $p$-subgroups of ${}^g\!DN$.  The size of this set does not depend on $g$ and by Sylow's third theorem, $|\tn{Syl}_p({}^g\!DN)| = 1$ in $k$.  Hence
	$$|{}^G\!D|_{p'} = |{}^{G/N}\!DN/N|_{p'}\cdot|\tn{Syl}_p(DN)|_{p'} = |{}^{G/N}\!DN/N|_{p'},$$
	as required.
\end{proof}

% \begin{lem}\label{lem:Brdiagramcommutes}
% 	Let $D$ be a pro-$p$ subgroup of $G$ that is open in a $p$-Sylow subgroup of $G$ that contains it.  
% 	The diagram
% 	$$\xymatrix{
% 		k\db{G}^D \ar[rr]^-{\tn{Br}_D}\ar[d]_{\widetilde{\tn{Tr}}_D^G} && k\db{\tn{C}_G(D)} \ar[d]^{\widetilde{\tn{Tr}}_D^{\tn{N}_G(D)}}  \\
% 		\widetilde{\tn{Tr}}_D^G(k\db{G}^D)\ar[rr]_-{\tn{Br}_D} && k\db{\tn{C}_G(D)}
% 	}$$
% 	commutes up to non-zero scalar multiple.  In particular the image of the lower map is precisely $\widetilde{\tn{Tr}}_D^{\tn{N}_G(D)}(k\db{\tn{C}_G(D)})$.
% \end{lem}

\begin{lem}\label{lem:Brdiagramcommutes}
	Let $D$ be a pro-$p$ subgroup of $G$ that is open in a $p$-Sylow subgroup of $G$ that contains it.  Let $N$ be an open normal subgroup of $G$ such that $D$ is a $p$-Sylow subgroup of $DN$, and denote by $b$ the element $1 / |G:\tn{N}_G(DN)|_{p'}$ of $k$.
	The diagram
	$$\xymatrix{
		k\db{G}^D \ar[rr]^-{\tn{Br}_D}\ar[d]_{\widetilde{\tn{Tr}}_D^G} && k\db{\tn{C}_G(D)} \ar[d]^{b\cdot\widetilde{\tn{Tr}}_D^{\tn{N}_G(D)}}  \\
		\widetilde{\tn{Tr}}_D^G(k\db{G}^D)\ar[rr]_-{\tn{Br}_D} && k\db{\tn{C}_G(D)}
	}$$
	commutes.  In particular the image of the lower map is precisely $\widetilde{\tn{Tr}}_D^{\tn{N}_G(D)}(k\db{\tn{C}_G(D)})$.
\end{lem}

\begin{proof}
	We work within a cofinal set $\mathcal{N}$ of open normal subgroups $N$ of $G$ such that $$|\tn{N}_G(D)N:DN|_p = |\tn{N}_G(D):D|_p\,\hbox{ and }\,|G:\tn{N}_G(D)N|_p = |G:\tn{N}_G(D)|_p$$ so that Lemma \ref{Lemma tildeTr properties} holds when required.  These conditions imply that $D$ is a $p$-Sylow subgroup of $DN$ and hence by Lemma \ref{lemma index independent of N}, the value $1/|G:\tn{N}_G(DN)|_{p'}$ is equal to $b$ for every $N\in \mathcal{N}$. 
	%Denote by $b$ the element $1 / |G:\tn{N}_G(DN)|_{p'}$ of $k$, for $N$ some element of $\mathcal{N}$.
	
	We first claim that the following diagram is well-defined and commutes:
	
	$$\xymatrix{
		k[G/N]^{DN/N} \ar[r]\ar@/^1.5pc/[rr]^{\tn{Br}_{DN/N}}\ar[d]_{\varphi_{MN}} & k[G/N]^{DN/N}/\sum_{QN<DN}\tn{Tr}_{QN/N}^{DN/N}(k[G/N]^{DN/N})\ar[r]\ar[d]^{\varphi_{MN}} & k[\tn{C}_{G/N}(DN/N)]\ar[d]^{\varphi_{MN}}  \\
		k[G/M]^{DM/M} \ar[r]\ar@/_1.5pc/[rr]_{\tn{Br}_{DM/M}} & k[G/M]^{DM/M}/\sum_{QM<DM}\tn{Tr}_{QM/M}^{DM/M}(k[G/M]^{DM/M})\ar[r] & k[\tn{C}_{G/M}(DM/M)]
	}$$
	wherein the symbol $\varphi_{MN}$ is abusively being used to indicate the map induced by $\varphi_{MN}:k[G/N]\to k[G/M]$ on the corresponding subquotients.  The outer triangles commute by definition, and as long as the middle vertical map is well-defined, then the inner squares also commute by definition.  So the only thing that needs to be checked is that the middle vertical map is well defined, and for this we must see what happens to an element $\tn{Tr}_{QN/N}^{DN/N}(a)$ when $QN$ is proper in $DN$, but $QM=DM$.  But in this case
	$$\varphi_{MN}(\tn{Tr}_{QN/N}^{DN/N}(a)) = \sum_{dN\in DN/QN}{}^{dM}\varphi_{MN}(a) = |DN:QN|\varphi_{MN}(a) = 0$$
	because $\varphi_{MN}(a)$ is a $DM/M$-fixed point and $|DN:QN| = 0$ in $k$.  The map is thus well-defined and the claim follows.  In particular, the outer shape commutes, and so $\varphi_{MN}\tn{Br}_{DN/N} = \tn{Br}_{DM/M}\varphi_{MN}$.
	
	Write $k\db{G}^D = k\db{\tn{C}_G(D)}\oplus \overline{\sum_{Q<_OD}\tn{Tr}_Q^D(k\db{G}^Q)}$. We first claim that the second summand is sent to $0$ when passed in either direction through the diagram of the statement of the lemma, so suppose that $x = \tn{Tr}_Q^D(a)$ for $Q$ a proper open subgroup of $D$ and $a\in k\db{G}^Q$.  Note that $\widetilde{\tn{Tr}}_D^G\tn{Tr}_Q^D(a) = \widetilde{\tn{Tr}}_Q^G(a)$, as can be seen by looking at the cofinal subset of those $N$ for which $QN\cap D = Q$.
	For such $N$ we have  
	$$\tn{Tr}_{QN/N}^{G/N}(\varphi_N(a)) = \sum_{g\in DN\backslash G/QN}\tn{Tr}_{({}^g\!QN\cap DN)/N}^{DN/N}({}^g\!\varphi_N(a)).$$
	The argument of Proposition \ref{prop:defgrpcharbyBr_D} shows that the subgroups $({}^g\!QN\cap DN)/N$ are proper in $DN/N$, and hence
	$$\tn{Tr}_{QN/N}^{G/N}(\varphi_N(a)) \in \tn{Ker}(\tn{Br}_{DN/N}).$$
	Thus $\tn{Br}_D\widetilde{\tn{Tr}}_D^G(x) = 0$.  But going the other way round the square, $b\cdot\widetilde{\tn{Tr}}_D^{\tn{N}_G(D)}\tn{Br}_D(x)$ is clearly $0$, so the claim follows.
	
	We from now on assume that $x\in k\db{\tn{C}_G(D)}$.  Denote by $s$ the element $\tn{Br}_D\widetilde{\tn{Tr}}_D^G(x)$ and by $t$ the element $\widetilde{\tn{Tr}}_D^{\tn{N}_G(D)}\tn{Br}_D(x) = \widetilde{\tn{Tr}}_D^{\tn{N}_G(D)}(x)$.  We will show that for each $M\in \mathcal{N}$, we have
	$$\varphi_M(s) = b\cdot\varphi_M(t),$$
	and hence that $s = b\cdot t$.  Since $b$ does not depend on $s$ or $t$, the lemma follows. 
	
	We first calculate $\varphi_N(s)$ for arbitrary $N\in \mathcal{N}$:
	\begin{align*}
		\varphi_N(s) & = \varphi_N\tn{Br}_D\widetilde{\tn{Tr}}_D^G(x) \\
		& = \tn{Br}_{DN/N}\varphi_N\widetilde{\tn{Tr}}_D^G(x) \\
		& = \frac{1}{|G:DN|_{p'}}\tn{Br}_{DN/N}\tn{Tr}_{DN/N}^{G/N}(\varphi_N(x)) \\
		& = \frac{1}{|G:DN|_{p'}}\tn{Br}_{DN/N}\bigg(\sum_{g\in DN\backslash G/DN}\tn{Tr}_{(DN\cap{}^g\!DN)/N}^{DN/N}({}^{gN}\varphi_N(x)) \bigg).
	\end{align*}
	Note that $DN\cap {}^g\!DN = (D\cap {}^g\!DN)N$ by Dedekind's modular law applied to $D$ and $N\leqslant {}^g\!DN$.  We claim that if $g\not\in \tn{N}_G(DN)$, then $\tn{Tr}_{(DN\cap {}^g\!DN)/N}^{DN/N}({}^{gN}\varphi_N(x))$ is sent to $0$ by $\tn{Br}_{DN/N}$: observe that
	$$D\subseteq {}^g\!DN \Longleftrightarrow g\in \tn{N}_G(DN)$$
	and hence for our $g\not\in \tn{N}_G(DN)$, the subgroup $Q = D\cap {}^g\!DN$ is proper in $D$.  Note further that 
	$$\left|\frac{DN}{QN}\right| = \left|\frac{DQN}{QN}\right| = \left|\frac{D}{D\cap (D\cap {}^g\!DN)N}\right| = \left|\frac{D}{D\cap {}^g\!DN}\right| = \left|\frac{D}{Q}\right|.$$
	Hence $\tn{Tr}_{(DN\cap {}^g\!DN)/N}^{DN/N}({}^{gN}\varphi_N(x)) = \tn{Tr}_{QN/N}^{DN/N}({}^{gN}\varphi_N(x))$ 
	is in the kernel of $\tn{Br}_{DN/N}$ as claimed, and so continuing the above calculation we have
	\begin{align*}
		\varphi_N(s)
		& = \frac{1}{|G:DN|_{p'}} \tn{Br}_{DN/N}\bigg(\sum_{g\in \tn{N}_G(DN)/DN}\tn{Tr}_{(DN\cap{}^g\!DN)/N}^{DN/N}({}^{gN}\varphi_N(x)) \bigg) \\
		& = \frac{1}{|G:DN|_{p'}} \tn{Br}_{DN/N}\tn{Tr}_{DN/N}^{\tn{N}_G(DN)/N}(\varphi_N(x)) \\
		& =   \frac{1}{|G:\tn{N}_G(DN)|_{p'}} \tn{Br}_{DN/N}\widetilde{\tn{Tr}}_{DN/N}^{\tn{N}_G(DN)/N}(\varphi_N(x)).\qquad\qquad\qquad(*)
	\end{align*}
	
	Meanwhile
	$$\varphi_N(t) = \widetilde{\tn{Tr}}_{DN/N}^{\tn{N}_G(D)N/N}(\varphi_N(x)).$$
	
	Fix $M\in \mathcal{N}$.  Since $M$ is open and $\bigcap_{N\in \mathcal{N}}\tn{N}_G(DN) = \tn{N}_G(D)$, there is $N\in \mathcal{N}$ contained in $M$ with the property that $\tn{N}_G(DN)M = \tn{N}_G(D)M$.  Hence
	$$\varphi_{MN}\left(\tn{Tr}_{\tn{N}_G(D)N/N}^{\tn{N}_G(DN)/N}(\varphi_N(t))\right)
	= \left|\frac{\tn{N}_G(DN)}{\tn{N}_G(D)N}\right|\varphi_M(t),$$
	or in other words,
	$$\varphi_{MN}\widetilde{\tn{Tr}}_{\tn{N}_G(D)N/N}^{\tn{N}_G(DN)/N}(\varphi_N(t))
	= \varphi_M(t).\qquad\qquad\qquad(**)$$
	We put all these calculations together to obtain the claimed relationship between $\varphi_M(s)$ and $\varphi_M(t)$:
	\begin{align*}
		\varphi_M(s) & = \varphi_{MN}\varphi_N(s) \\
		& = \frac{1}{|G:\tn{N}_G(DN)|_{p'}}\cdot \varphi_{MN}\tn{Br}_{DN/N}\widetilde{\tn{Tr}}_{DN/N}^{\tn{N}_G(DN)/N}(\varphi_N(x))&&\hbox{by }(*)\\
		& = \frac{1}{|G:\tn{N}_G(DN)|_{p'}}\cdot \varphi_{MN}\tn{Br}_{DN/N}\widetilde{\tn{Tr}}_{\tn{N}_G(D)N/N}^{\tn{N}_G(DN)/N}\widetilde{\tn{Tr}}_{DN/N}^{\tn{N}_G(D)N/N}(\varphi_N(x))&&\hbox{by Lemma }\ref{Lemma tildeTr properties}\\
		& =  \frac{1}{|G:\tn{N}_G(DN)|_{p'}}\cdot \tn{Br}_{D/M}\varphi_{MN}\widetilde{\tn{Tr}}_{\tn{N}_G(D)N/N}^{\tn{N}_G(DN)/N}(\varphi_N(t))\\
		& =  \frac{1}{|G:\tn{N}_G(DN)|_{p'}}\cdot \tn{Br}_{D/M}\varphi_{M}(t)&&\hbox{by }(**)\\
		& = b\cdot \varphi_M(t).
	\end{align*}
\end{proof}

\begin{lem}\label{lem:centralidemincentralizer}
	Let $D$ be a closed normal pro-$p$ subgroup of $G$.  The central idempotents of $k\db{G}$ lie in $k\db{\tn{C}_{G}(D)}$.
\end{lem}

\begin{proof}
	Fix an idempotent $e\in \tn{Z}(k\db{G})$.  For each $N\trianglelefteq_{O}G$, $DN/N$ is a normal $p$-subgroup of $G/N$ and $\varphi_{N}(e)$ is a central idempotent of $k[G/N]$.  By the finite version of this result \cite[Proposition 6.2.2]{B}, $\varphi_{N}(e)\in k[\tn{C}_{G/N}(DN/N)]$, and hence 
	$e\in \varprojlim k[\tn{C}_{G/N}(DN/N)] = k\db{\tn{C}_G(D)}.$
\end{proof}

\begin{lem}\label{lem:existb^G}
	Let $D$ be a closed pro-$p$ subgroup of the profinite group $G$ and let $H$ be a closed subgroup of $G$ such that $D\tn{C}_{G}(D)\leqslant H\leqslant \tn{N}_{G}(D)$.  Let $b$ be a block of $H$ with block idempotent $f$ and defect group $D$.  There is a unique block $B$ of $G$  with block idempotent $e$ such that $f=f\cdot \tn{Br}_{D}(e)$.
\end{lem}

\begin{proof}
	By hypothesis $D$ is normal in $H$, and hence $f\in k\db{\tn{C}_H(D)}$ by Lemma \ref{lem:centralidemincentralizer}.  But again by hypothesis $\tn{C}_H(D) = \tn{C}_G(D)$, and hence $f\in k\db{\tn{C}_G(D)}$.
	
	The Brauer homomorphism is a surjective algebra homomorphism $k\db{G}^D\to k\db{\tn{C}_G(D)}$, and hence (because $\tn{Br}_D(1) = 1$) there is at least one block idempotent $e$ such that $f=f\cdot \tn{Br}_{D}(e)$.  If $e'$ is another block idempotent with this property, then
	$$f = f\cdot \tn{Br}_D(e)\cdot\tn{Br}_D(e') = f\cdot \tn{Br}_D(ee'),$$
	which implies that $e=e'$, since distinct block idempotents are orthogonal.
\end{proof}

\begin{defi}\label{def:blockcorrespbyidemp}
	Let $D$ be a closed pro-$p$ subgroup of $G$ and $H$ a closed subgroup of $G$ with $D\tn{C}_{G}(D)\leq H\leq \tn{N}_{G}(D)$. If $b$ is a block of $H$ with block idempotent $f$, define the \emph{Brauer correspondent} $b^{G}$ of $b$ to be the unique block of $G$ with block idempotent $e$ such that  $f=f\cdot \tn{Br}_{D}(e)$. 
\end{defi}

We may finally prove our main theorem -- Brauer's first main theorem for profinite groups. 

% \begin{teo}\label{teo:Brauercorrespforprofinitegroups}
% 	Let $G$ be a profinite group, $P$ a $p$-Sylow subgroup of $G$ and $D$ an open subgroup of $P$. The map $\Phi$ sending a block $b$ of $\tn{N}_{G}(D)$ with defect group $D$ to its Brauer correspondent $b^G$ induces a bijection between the blocks of $\tn{N}_{G}(D)$ with defect group $D$ and the blocks of $G$ with defect group $D$.
% \end{teo}

\begin{teo}\label{teo:Brauercorrespforprofinitegroups}
	Let $G$ be a profinite group, $P$ a $p$-Sylow subgroup of $G$ and $D$ an open subgroup of $P$. There is a bijective correspondence between the blocks of $\tn{N}_{G}(D)$ with defect group $D$ and the blocks of $G$ with defect group $D$, given by sending the block $b$ of $\tn{N}_{G}(D)$ with defect group $D$ to its Brauer correspondent $b^G$.	
%	The map $\Phi$ sending a block $b$ of $\tn{N}_{G}(D)$ with defect group $D$ to its Brauer correspondent $b^G$ induces a bijection between the blocks of $\tn{N}_{G}(D)$ with defect group $D$ and the blocks of $G$ with defect group $D$.
\end{teo}

\begin{proof}
	Given the results already obtained, we are in a position to proceed as with finite groups.  We mostly follow \cite[Theorem 12.6.4]{Web}.
	
	By Lemma \ref{lem:centralidemincentralizer}, the block idempotents of $k\db{\tn{N}_{G}(D)}$ lie in $k\db{\tn{C}_{\tn{N}_{G}(D)}(D)} = k\db{\tn{C}_{G}(D)}$.
	
	Let $b$ be a block of $\tn{N}_{G}(D)$ with block idempotent $f$ and defect group $D$.  By Lemma \ref{lem:existb^G} the Brauer correspondent $b^G$ exists.  Let $e$ be the block idempotent of $b^G$, so that $f = f\cdot \tn{Br}_D(e)$ by definition.  Since $\tn{Br}_D(e)\neq 0$, $b^G$ has a defect group $R$ containing $D$ by Theorem \ref{theorem:CompleteDefectGroupChar}.  We claim that $R=D$.
	
	The subspace $e\widetilde{\tn{Tr}}_{D}^{G}(k\db{G}^{D})$ is an ideal of $e\tn{Z}(k\db{G})$ by Part 3 of Lemma \ref{Lemma tildeTr properties}.  By Lemma \ref{lem:Brdiagramcommutes} and since $\tn{Br}_{D}$ is an algebra homomorphism,
	$$\tn{Br}_{D}\left(e\widetilde{\tn{Tr}}_{D}^{G}(k\db{G}^{D})\right)=\tn{Br}_{D}(e)\tn{Br}_{D}\left(\widetilde{\tn{Tr}}_{D}^{G}(k\db{G}^{D})\right)=\tn{Br}_{D}(e)\widetilde{\tn{Tr}}_{D}^{\tn{N}_{G}(D)}(k\db{\tn{C}_{G}(D)}),$$
	which contains $f$.  Note that $f\not\in J(k\db{C_{G}(D)})$, since if it were then $\varphi_N(f)$ would be a non-zero idempotent in $J(k\db{C_{G}(D)}_N)$ for some $N$, which is impossible since the radical of a finite dimensional algebra is nilpotent.  Hence $\tn{Br}_{D}\left(e\widetilde{\tn{Tr}}_{D}^{G}(k\db{G}^{D})\right)\not\subseteq J(k\db{C_{G}(D)})$.  It follows that $e\widetilde{\tn{Tr}}_{D}^{G}(k\db{G}^{D})\not\subseteq J(e\tn{Z}(k\db{G}))$.  But the idempotent $e$ is primitive in $\tn{Z}(k\db{G})$ so that $e\tn{Z}(k\db{G})$ is local by Lemma \ref{lem:Roslem}, and hence $e\tn{Z}(k\db{G}) = e\widetilde{\tn{Tr}}_{D}^{G}(k\db{G}^{D})=\widetilde{\tn{Tr}}_{D}^{G}(ek\db{G}^{D})$.  In particular $e\in \tn{Im}(\widetilde{\tn{Tr}}_{D}^{G})$, showing that $R$ is contained in $D$ by Theorem \ref{theorem:CompleteDefectGroupChar}.  So $R=D$ as claimed.
	
	For convenience denote by $\Phi$ the map from the blocks of $N_G(D)$ with defect group $D$ to the blocks of $G$ with defect group $D$, sending the block $b$ to its Brauer correspondent $b^G$. The argument above shows that $\Phi$ is well-defined, so it remains to check that it is bijective.
	
	Suppose that ${b_1}^G = {b_2}^G$ has block idempotent $e$.  Then, letting $f_i$ be the block idempotent of $b_i$ we have
	$$f_i = f_i\tn{Br}_D(e) \in \tn{Br}_D(e\tn{Z}(k\db{G})).$$
	But $\tn{Br}_D(e\tn{Z}(k\db{G}))$ is local and hence $f_1=f_2$, showing that $\Phi$ is injective.
	
	Let $B$ be a block of $G$ with defect group $D$ and block idempotent $e$, so that $\tn{Br}_{D}(e)\neq 0$ and $e\in \widetilde{\tn{Tr}}_{D}^{G}(k\db{G}^{D})$.  Let $f$ be a block idempotent of $k\db{\tn{N}_G(D)}$ such that $f\cdot\tn{Br}_D(e) = f$.  If $b = k\db{\tn{N}_G(D)}f$ has defect group $D$, then $B = \Phi(b)$ and $\Phi$ is surjective.  It thus remains to check that $b$ has defect group $D$.  By Lemma \ref{lem:Brdiagramcommutes}, $\tn{Br}_D(e) = \widetilde{\tn{Tr}}_D^{\tn{N}_G(D)}(a)$ for some $a$, so that by Lemma \ref{Lemma tildeTr properties},
	$$f = f\cdot \widetilde{\tn{Tr}}_D^{\tn{N}_G(D)}(a)
	= \widetilde{\tn{Tr}}_D^{\tn{N}_G(D)}(fa)$$
	and hence $b$ has a defect group $R$ contained in $D$ by Theorem \ref{theorem:CompleteDefectGroupChar}.  On the other hand, $R$ is an intersection of $p$-Sylow subgroups of $\tn{N}_G(D)$ by Proposition \ref{prop:defectgroupsintersectionofsylows}, while $D$ is the intersection of all the $p$-Sylow subgroups of $\tn{N}_G(D)$ by Corollary \ref{cor:defectgrpmaximalinnormalizer}.  Hence $D\leqslant R$ and we are done.
\end{proof}

\begin{cor}
	Let $G$ be a profinite group, $P$ a $p$-Sylow subgroup of $G$, $D$ an open subgroup of $P$ and $L$ a closed subgroup of $G$ containing $\tn{N}_G(D)$.  There is a canonical bijection between the blocks of $G$ with defect group $D$ and the blocks of $L$ with defect group $D$.
\end{cor}

\begin{proof}
	Since $\tn{N}_G(D) = \tn{N}_L(D)$, Theorem \ref{teo:Brauercorrespforprofinitegroups} applies for both $G$ and $L$, yielding the result. 
\end{proof}

\begin{rem}\label{Remark:BrauerCorrFails}
	Given that a defect group is open in a $p$-Sylow subgroup of $G$ that contains it by Proposition \ref{prop:defgrpopen}, Theorem \ref{teo:Brauercorrespforprofinitegroups} applies to arbitrary blocks of $G$.  It is never-the-less important to note that Theorem \ref{teo:Brauercorrespforprofinitegroups} is false when we replace the word ``open'' in its statement with the word ``closed'':  Let $G$ be the free pro-$p$ group of rank $2$, freely generated by $x$ and $y$.  If $D$ is the closed subgroup generated by $x$, then $D = \tn{N}_G(D)$.  There is of course a block of $k\db{\tn{N}_G(D)}$ with defect group $D$ (namely $k\db{\tn{N}_G(D)}$ itself), but there is not a block of $k\db{G}$ with defect group $D$ ($k\db{G}$ has only one block and its defect group is $G$). 
\end{rem}

\bibliography{refpap}

\begin{thebibliography}{10}

\bibitem{A}
Jonathan~L. Alperin.
\newblock {\em Local representation theory: Modular representations as an
  introduction to the local representation theory of finite groups}, volume~11.
\newblock Cambridge University Press, 1993.

\bibitem{B}
David~J. Benson.
\newblock {\em Representations and Cohomology}, volume~1 of {\em Cambridge
  Studies in Advanced Mathematics}.
\newblock Cambridge University Press, 1991.

\bibitem{brauer1944}
Richard Brauer.
\newblock On the arithmetic in a group ring.
\newblock {\em Proceedings of the National Academy of Sciences of the United
  States of America}, 30(5):109, 1944.

\bibitem{Bru}
Armand Brumer.
\newblock Pseudocompact algebras, profinite groups and class formations.
\newblock {\em Journal of Algebra}, 4(3):442--470, 1966.

\bibitem{Craven}
David~A. Craven.
\newblock {\em Representation theory of finite groups: a guidebook}.
\newblock Universitext. Springer, Cham, 2019.

\bibitem{dimitric}
Radoslav Dimitric.
\newblock {\em Slenderness}, volume~1.
\newblock Cambridge University Press, 2018.

\bibitem{G}
Pierre Gabriel.
\newblock Des cat\'egories ab\'eliennes.
\newblock {\em Bulletin de la Soci\'et\'e Math\'ematique de France},
  90:323--448, 1962.

\bibitem{Gr}
James~A. Green.
\newblock On the indecomposable representations of a finite group.
\newblock {\em Mathematische Zeitschrift}, 70:430--445, 1958.

\bibitem{JK}
Kostiantyn Iusenko and John~W. MacQuarrie.
\newblock The path algebra as a left adjoint functor.
\newblock {\em Algebras and Representation Theory}, 2017.

\bibitem{Knorr}
Reinhard Kn{\"o}rr.
\newblock Blocks, vertices and normal subgroups.
\newblock {\em Mathematische Zeitschrift}, 148(1):53--60, 1976.

\bibitem{L}
Solomon Lefschetz.
\newblock {\em Algebraic topology}, volume~27.
\newblock American Mathematical Soc., 1942.

\bibitem{Lin1}
Markus Linckelmann.
\newblock {\em The Block Theory of Finite Group Algebras}, volume~1.
\newblock Cambridge University Press, 2018.

\bibitem{Lin2}
Markus Linckelmann.
\newblock {\em The Block Theory of Finite Group Algebras}, volume~2.
\newblock Cambridge University Press, 2018.

\bibitem{J1}
John~W. MacQuarrie.
\newblock Green correspondence for virtually pro-$p$ groups.
\newblock {\em Journal of Algebra}, 323:2203--2208, 04 2010.

\bibitem{J}
John~W. MacQuarrie.
\newblock Modular representations of profinite groups.
\newblock {\em Journal of Pure and Applied Algebra}, 215(5):753--763, 2011.

\bibitem{PJ}
John~W. MacQuarrie and Peter Symonds.
\newblock Brauer theory for profinite groups.
\newblock {\em Journal of Algebra}, 398:496--508, 2014.

\bibitem{PPJ}
John~W. MacQuarrie, Peter Symonds, and Pavel~A. Zalesskii.
\newblock Infinitely generated pseudocompact modules for finite groups and
  {W}eiss' theorem.
\newblock {\em Advances in Mathematics}, 361:106925, 2020.

\bibitem{ZR}
Luis Ribes and Pavel~A. Zalesskii.
\newblock {\em Profinite groups}.
\newblock Springer, 2000.

\bibitem{P2}
Peter Symonds.
\newblock Permutation complexes for profinite groups.
\newblock {\em Commentarii Mathematici Helvetici}, 82(1):1--37, 2007.

\bibitem{VanGas}
Martine {Van Gastel} and Michel {Van den Bergh}.
\newblock Graded modules of {G}elfand–{K}irillov dimension one over
  three-dimensional {A}rtin–{S}chelter regular algebras.
\newblock {\em Journal of Algebra}, 196(1):251 -- 282, 1997.

\bibitem{War}
Seth Warner.
\newblock {\em Topological rings}.
\newblock Elsevier, 1993.

\bibitem{Web}
Peter Webb.
\newblock {\em A course in finite group representation theory}, volume 161.
\newblock Cambridge University Press, 2016.

\bibitem{W}
John~S. Wilson.
\newblock {\em Profinite groups}.
\newblock Oxford University Press, USA, 1999.

\end{thebibliography}

\end{document}